\documentclass[10pt,twoside,a4paper,english]{article}

\title{Shallow water asymptotic models for the propagation of internal waves}

\author{Vincent Duch\^ene%
\thanks{IRMAR - UMR6625, Univ. Rennes 1, CNRS, Campus de Beaulieu, F-35042 Rennes cedex, France. \url{vincent.duchene@univ-rennes1.fr}}, 
Samer Israwi%
\thanks{ Laboratory of Mathematics-EDST and Faculty of Sciences I, Lebanese University, Lebanon and CRAMS: Center for Research in Applied Mathematics and Statistics, AUL, Lebanon.  \url{s_israwi83@hotmail.com}},
Raafat Talhouk%
\thanks{Laboratoire de math\'ematiques-EDST et Facult\'e des sciences I, Universit\'e Libanaise, Beyrouth, Liban. \url{rtalhouk@ul.edu.lb}}
}

\date{\today}

\usepackage{babel}         
\usepackage[utf8]{inputenc} 
\usepackage[T1]{fontenc}    
\usepackage[ec]{aeguill}    
\usepackage{amsmath, amsthm} 
\usepackage{amsfonts,amssymb}
\usepackage{stmaryrd}
\usepackage{float}           
\usepackage[pdftex]{graphicx} 
\usepackage{epstopdf}
\usepackage{tabularx}
\usepackage{fancyhdr}       
\usepackage{url}            
\usepackage{hyperref}       

\numberwithin{equation}{section}

\fancypagestyle{thestyle}{%
	\fancyfoot{}
	\fancyhead[CE]{Shallow water asymptotic models for the propagation of internal waves} 
	\fancyhead[CO]{\scriptsize Vincent Duch\^ene, Samer Israwi, Raafat Talhouk} 
	\fancyhead[RE,LO]{\scriptsize\today}  
	\fancyhead[LE,RO]{\scriptsize\thepage}
	}

\newcommand{\RR}{\mathbb{R}}

\renewcommand{\t}{\widetilde}
\renewcommand{\b}{\overline}

\renewcommand{\P}{\mathcal{P}}
\newcommand{\N}{\mathcal{N}}
\newcommand{\T}{\mathcal{T}}
\newcommand{\Q}{\mathcal{Q}}
\newcommand{\R}{\mathcal{R}}
\renewcommand{\O}{\mathcal{O}}
\newcommand{\G}{\mathcal{G}}
\newcommand{\V}{\mathcal{V}}
\newcommand{\mfT}{\mathfrak{T}}
\newcommand{\p}{\mathfrak{p}}
\newcommand{\mfQ}{\mathfrak{Q}}
\newcommand{\mfR}{\mathfrak{R}}
\DeclareMathOperator*{\esssup}{ess\,sup}
\DeclareMathOperator{\Bo}{Bo}
\DeclareMathOperator{\bo}{bo}
\DeclareMathOperator{\curl}{curl}
\newcommand{\dsp}{\displaystyle}
\newcommand{\nn}{\nonumber}
\newcommand{\id}[1]{\left\vert_{\scriptstyle #1}\right.}

\newtheorem{Theorem}{Theorem}[section]
\newtheorem{Definition}[Theorem]{Definition}
\newtheorem{Proposition}[Theorem]{Proposition}
\newtheorem{Lemma}[Theorem]{Lemma}

\newtheorem{Remark}[Theorem]{Remark}

\begin{document}
\maketitle

\begin{abstract}
We are interested in asymptotic models for the propagation of internal waves at the interface between two shallow layers of immiscible fluid, under the rigid-lid assumption. We review and complete existing works in the literature, in order to offer a unified and comprehensive exposition. Anterior models such as the shallow water and Boussinesq systems, as well as unidirectional models of Camassa-Holm type, are shown to descend from a broad Green-Naghdi model, that we introduce and justify in the sense of consistency. Contrarily to earlier works, our Green-Naghdi model allows a non-flat topography, and horizontal dimension $d=2$. Its derivation follows directly from classical results concerning the one-layer case, and we believe such strategy may be used to construct interesting models in different regimes than the shallow-water/shallow-water studied in the present work.
\end{abstract}

%\tableofcontents

\section{Introduction}

The study of gravity waves at the surface of a homogeneous layer of fluid has attracted a lot of interests in a broad range of scientific communities. We let the reader refer to~\cite{Lannes} for a comprehensive survey of the state of the art concerning this problem, and its many interesting aspects, and we quickly discuss here some known results and methods, relevant to the present work.

 While the equations governing the motion of a homogeneous layer of ideal, incompressible, irrotationnal fluid under the only influence of gravity, that we name {\em full Euler system}, are relatively easy to derive, their theoretical study is extremely challenging.
This explains why the rigorous, mathematical analysis of the governing equations is quite recent, and still enjoys present-day improvements from an active community.  In particular, the well-posedness of the Cauchy problem outside of the analytical framework has been discussed among others by Nalimov~\cite{Nalimov74}, Yosihara~\cite{Yosihara82}, Craig~\cite{Craig85}, Wu~\cite{Wu97,Wu99} and Lannes~\cite{Lannes05}. Such results are regularly improved (time of existence, regularity of the initial data, {\em etc.}); see~\cite{Wu09,Wu11,GermainMasmoudiShatah12,AlazardBurqZuily12,IonescuPusateri13,AlazardDelort13} and references therein.

Nevertheless, the solutions of these equations are very difficult to describe, and the relevant  hydrodynamic processes are not easily visible in these equations. At this point, a classical method is to select an asymptotic regime (described by dimensionless parameters of the domain and of the flow), in which we look for approximate models and hence for approximate solutions.

Many such asymptotic models have been derived, going back to the  late 19\textsuperscript{th} century. For example, Saint-Venant~\cite{Saint-Venant71} derived the classical shallow-water equation, by assuming that the depth of the layer of fluid is small, so that the horizontal velocity across the layer may be averaged as a constant; while Boussinesq~\cite{Boussinesq71,Boussinesq72} derived the model which bears his name, describing the propagation of gravity-waves of small amplitude and long wavelength. Later on, Serre~\cite{Serre53} and Green,Naghdi~\cite{GreenNaghdi76} introduced a higher order model, which has since been widely used in coastal oceanography, as it takes into account the dispersive effects neglected by the shallow-water (Saint-Venant) model and allows waves of greater amplitude than the Boussinesq model.

However, the previously mentioned works are restricted to the formal level, and the rigorous, mathematical justification of asymptotic models received a satisfactory answer only recently. We say that a model is {\em fully justified} (using the terminology of~\cite{Lannes}) if the Cauchy problem for both the full Euler system and the asymptotic model is well-posed for a given class of initial data, and over the relevant time scale; and if the solutions with corresponding initial data remain close.
The full justification of  a system (S) follows from:
\begin{itemize}
\item (Consistency) One proves that families of solutions to the asymptotic model, existing and controlled over the relevant time scale satisfies the full Euler system up to a small residual. 
\item (Existence) One proves that solutions of the full Euler system and solutions of the the model~(S) with corresponding initial data do exist. 
\item (Convergence) One proves that the solutions of the full Euler system, and the ones of the asymptotic model, with corresponding initial data, remain close over the relevant time scale. 
\end{itemize}
A result of Alvarez-Samaniego and Lannes~\cite{Alvarez-SamaniegoLannes08} provides the existence and uniqueness of a solution to the full Euler system over the relevant time scale, uniformly with respect to the dimensionless parameters at stake, as well as a stability result with respect to perturbation of the equations. As a consequence, any consistent and well-posed asymptotic model is automatically justified in the sense described above. 
This strategy may be applied to the (quasilinear, hyperbolic) shallow-water model, to various Boussinesq models~\cite{KanoNishida86,BonaChenSaut02,BonaChenSaut04,Chazel07,SautXu12} and to Green-Naghdi models~\cite{Alvarez-SamaniegoLannes08,Israwi10a,Israwi11}. We let the reader refer to~\cite[Appendix C]{Lannes} for a reader's digest of the numerous known results on this aspect.

All the aforementioned models coincide at the lower order of precision as a simple wave equation, which in dimension $d=1$ predicts that any initial perturbation of the free surface will split up into two counter-propagating waves. An important family of models is dedicated to the precise study of the evolution of one of the two waves when higher order terms are included. The most famous example of such model is the Korteweg-de Vries\cite{KortewegDe95} equation, but various extensions and generalizations have been proposed; see~\cite{Johnson02,ConstantinLannes09,Israwi10} and references therein. Again, the full justification of such models is recent: see~\cite{KanoNishida86,SchneiderWayne00,BonaColinLannes05,Chazel09,ConstantinLannes09,Israwi10}.
\bigskip

All the aforementioned works are concerned with the case of a single layer of homogeneous fluid. Such assumption of may seem too crude for applications to oceanographic problems, as variations of salinity induce variation of density. In the present work, we are interested in the simplest setting that models such a variation of density: we consider a system of two layers of homogeneous, immiscible fluid, and we are interested in the evolution of the interface between the two layers. A considerable amount of interests has been given to such bi-fluidic systems; see~\cite{HelfrichMelville06} for a comprehensive review of the ins and outs on this topic.

The governing equations of the bi-fluidic systems share many aspects and properties of the aforementioned water-wave system, and its study has often been carried out in parallel. In particular, one can derive asymptotic models in analogy with the ones presented above. It is out of the scope of this introduction to present an exhaustive review of all the different models, as many settings are of interests. Here, and in the present work, we restrict ourselves to the case of a surface delimited by a flat, rigid lid (as deformation of the surface is in practice small compared to the deformation of the interface), and to the so-called shallow water/shallow water regime. In this regime, the two layers are assumed to be of comparable depth, and both small when compared to the typical horizontal wavelength of the flow. In that case, the models corresponding to the shallow water and Boussinesq systems have been derived in~\cite{ChoiCamassa96,CraigGuyenneKalisch05}, and justified in the sense of consistency in~\cite{BonaLannesSaut08} (where, incidentally, a much larger range of scaling regimes are studied). Green-Naghdi type models where obtained in the one-dimensional case in~\cite{Matsuno93}, and in the two-dimensional in~\cite{ChoiCamassa99}. An extensive study of scalar models has been provided by one of the authors in~\cite{Duchene13}. Let us note that all the aforementioned works are restricted to the case of a flat bottom, contrarily to the present work.

As attested earlier, bi-fluidic models have a similar structure as the ones in the one-layer case. As a matter of fact, one recovers the latter from the former when we assume that the mass density of the top layer is zero. Yet a few remarkable differences arise, that originate interesting questions and mathematical challenges. Among them, we would like to emphasize
\begin{enumerate}
\item {\em The role of surface tension.} Contrarily to the water wave case, the Cauchy problem for the full Euler system is ill-posed in Sobolev spaces in the absence of surface tension. However, surface tension is very small in practical cases, so that its effect is systematically negligible in all present asymptotic models. In~\cite{Lannes13}, Lannes shows that a small amount of surface tension is sufficient to guarantee the well-posedness over times consistent with observations, provided that a stability estimate holds; see the somewhat more precise description in Section~\ref{sec:WPfullEuler}.
\item {\em Absence of stability result.} An equivalent result as the stability result (with respect to perturbations of the equation) for the full Euler system as described above is not known in the bi-fluidic case, partly due to the difficulties described in the previous item. In order to deal with this, a strategy consists in proving such stability result on the model itself. Thus the full justification of the model is a consequence of its well-posedness, and the full Euler system's consistency with the asymptotic model (and not the other way around). This strategy has been applied by the authors in~\cite{DucheneIsrawiTalhouk}, and we recall these results in Section~\ref{sec:Serre}. 
\item {\em A non-local operator.} As noticed in~\cite{BonaLannesSaut08}, shallow water models for internal waves with a rigid lid contain a non-local operator, which involves in particular the projection into the space of gradient functions; see Definition~\ref{defV} in Section~\ref{sec:GN}. This operator appears only in the case $d=2$, and under the rigid-lid configuration (see~\cite{Duchene10} for shallow water/shallow water models with a free surface). The precise effect and meaning of this non-local term is yet to be fully understood.
\item {\em A critical-ratio.} There exists a critical ratio for the depth of the two layers for which the first order (quadratic) nonlinearities vanish; $\delta^2=\gamma$ in, {\em e.g.},~\eqref{eqn:BoussinesqMean}. This phenomenon does not occur in the one-layer case, and 
%occurs only for one of the two modes in the bi-fluidic case with a free surface (see~\cite{Duchene11a}). This 
motivates a precise study of unidirectional asymptotic models with stronger nonlinearities than in the classical long wave regime, and especially in the Camassa-Holm regime, for which first order dispersion and nonlinearities are formally of same magnitude in the critical case. This study has been carried out in~\cite{Duchene13}, completed in~\cite{DucheneIsrawiTalhouk}, and the results are presented in Section~\ref{sec:scalar}.
\end{enumerate}

In the present work, we report and complete recent mathematical results concerning the bi-fluidic system under the rigid lid assumption; from the well-posedness of the full Euler system to the justification of various models in the shallow water regime. Our aim is to provide a unified and comprehensive exposition of the existing theory. The above concerns and remarks appear spontaneously in the course of the study. We would like to mention the work of Saut, which pursues similar objectives than ourselves in~\cite{Saut12}.

\paragraph{Organization of the paper.}
The present paper is organized as follows.
In Section~\ref{sec:fullEuler}, we introduce the non-dimensionalized full Euler equations describing the evolution of the two-fluid system with a rigid-lid we consider. We roughly describe in Theorem~\ref{th:WPEuler} its well-posedness result, obtained by Lannes in~\cite{Lannes13} (for a flat topography).
Section~\ref{sec:GNGN} is dedicated to the construction and justification (in the sense of consistency) of the Green-Naghdi models. This result, in dimension $d=2$, and allowing non-flat topography, is new to our knowledge. Several lower order models, for which stronger results have been recently obtained, are shown to descend directly from our Green-Naghdi system, and are described in Section~\ref{sec:lowerorder}. More precisely:
\begin{itemize}
\item Section~\ref{sec:shallowwater}: The shallow water (Saint Venant) model, introduced in~\cite{BonaLannesSaut08} and studied in details in~\cite{GuyenneLannesSaut10};
\item Section~\ref{sec:Serre}: A very recent Green-Naghdi model in the Camassa-Holm regime, introduced and rigorously justified in~\cite{DucheneIsrawiTalhouk};
\item Section~\ref{sec:Boussinesq}: Boussinesq models, whose study follows from results in~\cite{BonaColinLannes05,Duchene11a}, and that we adapt to our case;
\item Section~\ref{sec:scalar}: Unidirectional (scalar) models generalizing the classical Korteweg-de Vries equation, whose rigorous justification  has been investigated in~\cite{Duchene13}.
\end{itemize}

\paragraph{Notations.}
 In the following, $C_0$ denotes any nonnegative constant whose exact expression is of no importance. The notation $a\lesssim b$ means that 
 $a\leq C_0 b$.\\
 We denote by $C(\lambda_1, \lambda_2,\dots)$ a nonnegative constant depending on the parameters
 $\lambda_1$, $\lambda_2$,\dots and whose dependence on the $\lambda_j$ is always assumed to be nondecreasing.\\
 Let $p$ be any constant with $1\leq p< \infty$. We denote $L^p=L^p(\RR^d)$ the space of all Lebesgue-measurable functions
 $f$ with the standard norm 
 \[\big\vert f \big\vert_{L^p}=\Big(\int_{\RR^d}\vert f(X)\vert^p dX\Big)^{1/p}<\infty.\]
  The real inner product of any functions $f_1$
 and $f_2$ in the Hilbert space $L^2(\RR^d)$ is denoted by
\[
 \big(\ f_1\ ,\ f_2\ \big)\ =\ \int_{\RR^d}f_1(X)f_2(X) dX.
 \]
 The space $L^\infty=L^\infty(\RR^d)$ consists of all essentially bounded, Lebesgue-measurable functions
 $f$ with the norm
\[
 \big\vert f\big\vert_{L^\infty}\ =\  \esssup_{X\in\RR^d} \vert f(X)\vert\ <\ \infty\ .
\]
 For any real constant $s\geq0$, $H^s=H^s(\RR^d)$ denotes the Sobolev space of all tempered
 distributions $f$ with the norm $\vert f\vert_{H^s}=\vert \Lambda^s f\vert_2 < \infty$, where $\Lambda$
 is the pseudo-differential operator $\Lambda=(1-\Delta)^{1/2}$.\\
For convenience, we will make use of the following notation for given $h_0,s,t_0\geq0$:
 \[ M(s) \ = \ C(\frac{1}{h_0},\big\vert \zeta \big\vert_{H^{\max(s,t_0+1)}},\big\vert b \big\vert_{H^{\max(s,t_0+1)}}) \ .\]
For a vector-valued function $F=(f_1,\dots f_n)^\top$, we write $F\in L^p(\RR^d)^n$ (resp. $F\in H^s(\RR^d)^n$) if each of the components $f_i\in L^p(\RR^d)$ (resp. $f_i\in H^s(\RR^d)$). The function spaces are endowed with canonical norms:
\[ \big\vert F \big\vert _{L^p} \ = \ \sum_{i=1}^n \big\vert f_i \big\vert_{L^p} \qquad \text{ and } \qquad \big\vert F \big\vert _{H^s} \ = \ \sum_{i=1}^n \big\vert f_i \big\vert_{H^s} \ .\]
 For any functions $u=u(t,X)$ and $v(t,X)$
 defined on $ [0,T)\times \RR^d$ with $T>0$, we denote the inner product, the $L^p$-norm and especially
 the $L^2$-norm, as well as the Sobolev norm,
 with respect to the spatial variable $X$, by $\big(u,v\big)=\big(u(t,\cdot),v(t,\cdot)\big)$, $\big\vert u \big\vert_{L^p}=\big\vert u(t,\cdot)\big\vert_{L^p}$, and $ \vert u \vert_{H^s}=\vert u(t,\cdot)\vert_{H^s}$, respectively.\\
  We denote $L^\infty([0,T);H^s)$ the space of functions such that $u(t,\cdot)$ is controlled in $H^s$, uniformly for $t\in[0,T)$:
 \[\big\Vert u\big\Vert_{L^\infty([0,T);H^s)} \ = \ \esssup_{t\in[0,T)}\vert u(t,\cdot)\vert_{H^s} \ < \ \infty.\]
 Finally, $C^k(\RR^d)$ denote the space of
 $k$-times continuously differentiable functions.
  \medskip

We conclude this section by the nomenclature that we use to describe the different regimes that appear in the present work. A regime is defined through restrictions on the set of admissible dimensionless parameters of the system, which are precisely defined in~\eqref{parameters}, below.
\begin{Definition}[Regimes] \label{Def:Regimes}
We designate by {\em shallow water regime} the set of parameters
\[
\P_{SW}  \equiv  \big\{ (\mu,\epsilon,\delta,\gamma),\ 0 <  \mu  \leq  \mu_{\max}, \ 0  \leq  \epsilon  \leq  1, \ \delta_{\min} \leq \delta \leq  \delta_{\max}, \ 0 \leq \gamma < 1  \big\},\]
with fixed $0<\mu_{\max},\delta_{\min},\delta_{\max}<\infty$.
We designate by {\em Camassa-Holm} regime (see~\cite{ConstantinLannes09}) the set
\[
\P_{CH} \ \equiv \ \P_{SW} \ \cap\ \big\{ (\mu,\epsilon,\delta,\gamma),\ 0 \ \leq \ \epsilon \ \leq \ M\sqrt{\mu} \ \big\},\]
and by {\em long wave regime} the set 
\[
\P_{LW} \ \equiv \ \P_{SW} \ \cap\ \big\{ (\mu,\epsilon,\delta,\gamma),\ 0 \ \leq \ \epsilon \ \leq \ M\mu \ \big\},\]
with some fixed $M\in(0,\infty)$.

Additionally, unless otherwise indicated, it is assumed that $\Bo^{-1}\in[0,\Bo_{\min}^{-1}]$ and $\beta\in[0,1]$. The dependency of the constants on $\Bo_{\min}^{-1}$ is not displayed ($\Bo_{\min}^{-1}\leq1$ in any oceanographic application).
\end{Definition}

\section{The full Euler system}\label{sec:fullEuler}
\subsection{Construction of the full Euler system}
The system we study consists in two layers of immiscible, homogeneous, ideal, incompressible fluid under the only influence of gravity. The two layers are infinite in the horizontal dimension, and delimited above by a flat rigid lid and below by a non-necessarily flat bottom. The derivation of the governing equations of such a system is not new. We briefly recall it below, and refer to~\cite{BonaLannesSaut08,Anh09,Duchene13} for more details.

\begin{figure}[!htb]
\centering
 \includegraphics[width=\textwidth]{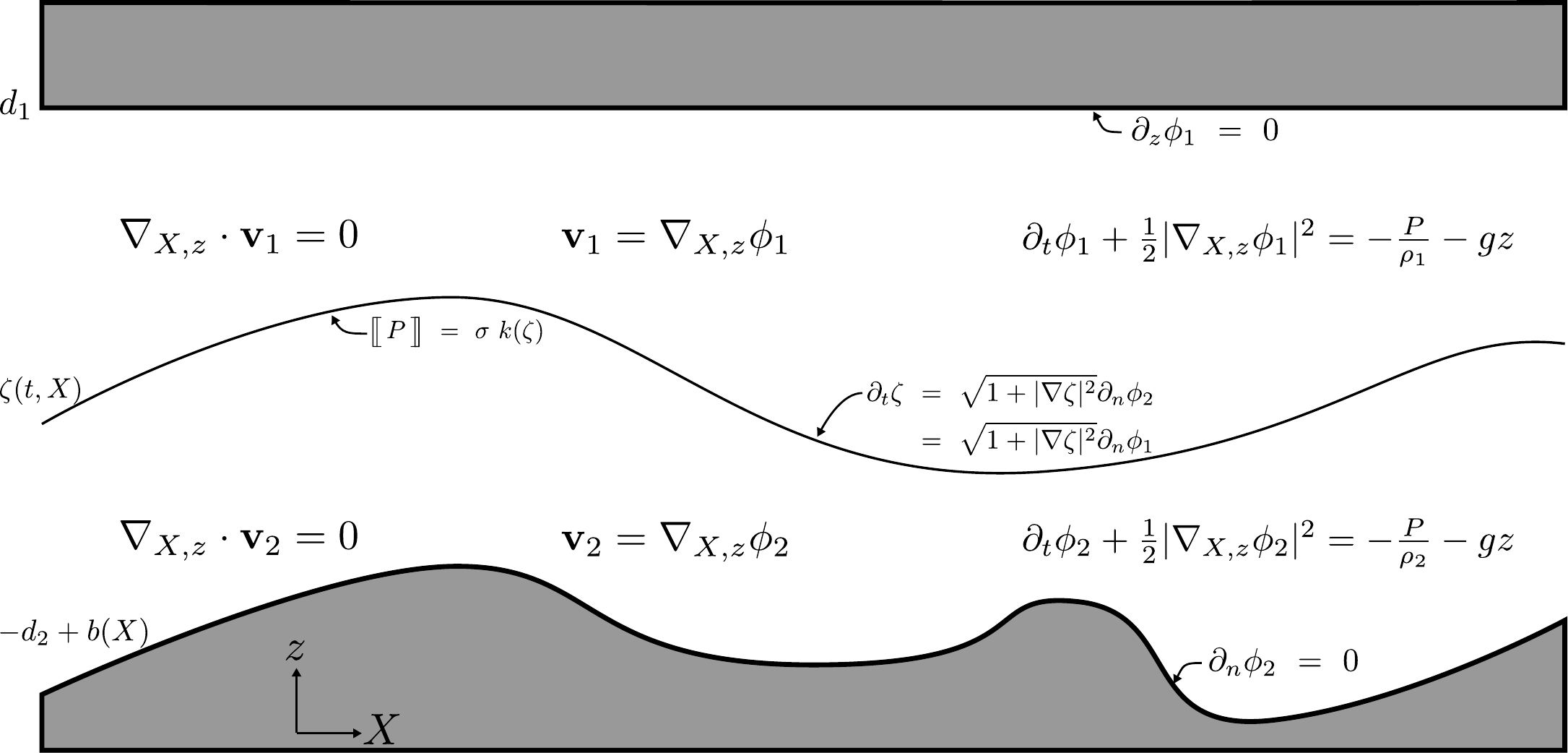}
 \caption{Sketch of the domain and governing equations}
\label{fig:SketchOfDomain}
\end{figure}

We assume that the interface and bottom are given as the graph of a function (resp. $\zeta(t,X)$ and $b(X)$) which expresses the deviation from their rest position (resp. $(X,0)$,$(X,-d_2)$) at the spatial coordinate $X\in \RR^d$ ($d=1$ or $d=2$) and at time $t$. 
Therefore, at each time $t\ge 0$, the domains of the upper and lower fluid (denoted, respectively, $\Omega_1^t$ and $\Omega_2^t$), are given by
\begin{align*}
 \Omega_1^t \ &= \ \{\ (X,z)\in\RR^d\times\RR, \quad \zeta(t,X)\ \leq\ z\ \leq \ d_1\ \}, \\
 \Omega_2^t \ &= \ \{\ (X,z)\in\RR^d\times\RR, \quad -d_2+b(X) \  \leq\ z\ \leq \ \zeta(t,X)\ \}.
\end{align*}
We assume that the two domains are strictly connected, that is there exists $h>0$ such that
\[ d_1- \zeta(t,X) \geq h>0, \quad \text{and} \quad d_2+\zeta(t,X)-b(X)\geq h>0.\]

We denote by $(\rho_1,{\mathbf v}_1)$ and $(\rho_2,{\mathbf v}_2)$ the mass density and velocity fields of, respectively, the upper and  the lower fluid. The two fluids are assumed to be homogeneous and incompressible, so that the mass densities $\rho_1,\ \rho_2$ are constant, and the velocity fields ${\mathbf v}_1,\ {\mathbf v}_2$ are divergence free.

Here and thereafter, we use the notation $\nabla_{X,z}$ to designate  the gradient operator with respect to the variable $(X, z)$, while the $\nabla$ and $\Delta$ simply denote the gradient and the Laplacian operators with respect to the variable $X$.

As we assume the flows to be irrotational, one can express the velocity field as gradients of a potential: ${\mathbf v}_i=\nabla_{X,z}\phi_i$ $(i=1,2)$ , and the velocity potentials satisfy Laplace's equation
\[\Delta \phi_i \ + \ \partial_z^2 \phi_i \ = \ 0  \quad (i=1,2).\]

 The fluids being ideal, they satisfy the Euler equations. Integrating the momentum equations yields Bernoulli equations, written in terms of the velocity potentials:
 \[   \partial_t \phi_i+\frac{1}{2} |\nabla_{X,z} \phi_i|^2=-\frac{P}{\rho_i}-gz \quad \text{  in }\Omega^t_i \quad (i=1,2),\]
 where $P$ denotes the pressure inside the fluid.

 From the assumption that no fluid particle crosses the surface, the bottom or the interface, one deduces kinematic boundary conditions, and the set of equations is closed by the continuity of the stress tensor at the interface, which reads
\[  \llbracket P(t,X) \rrbracket \ \equiv \ \lim\limits_{\varepsilon\to 0} \Big( \ P(t,X,\zeta(t,X)+\varepsilon) \ - \ P(t,X,\zeta(t,X)-\varepsilon) \ \Big) \ = \ -\sigma k\big(\zeta(t,X)\big),\]
where $ k(\zeta)=-\nabla \cdot \Big(\frac1{\sqrt{1+|\nabla\zeta|^2}}\nabla\zeta\Big)$ denotes the mean curvature of the interface, and $\sigma$ is the surface tension coefficient.

Altogether, the governing equations of our problem are given by the following
  \begin{equation}  \label{eqn:EulerComplet}
\left\{\begin{array}{ll}
         \Delta \phi_i \ + \ \partial_z^2 \phi_i \ = \ 0 & \mbox{ in }\Omega^t_i, \ i=1,2,\\
         \partial_t \phi_i+\frac{1}{2} |\nabla_{X,z} \phi_i|^2=-\frac{P}{\rho_i}-gz & \mbox{ in }\Omega^t_i, \ i=1,2, \\
         \partial_{z}\phi_1 \ = \ 0  & \mbox{ on } \Gamma_{\text{top}}\equiv\{(X,z),z=d_1)\}, \\
         \partial_t \zeta  \ = \ \sqrt{1+|\nabla\zeta|^2}\partial_{n}\phi_1 \ = \ \sqrt{1+|\nabla\zeta|^2}\partial_{n}\phi_2  & \mbox{ on } \Gamma \equiv\{(X,z),z=\zeta(t,X)\},\\
         \partial_{n}\phi_2  \ = \ 0 &  \mbox{ on } \Gamma_{\text{bot}}\equiv\{(X,z),z=-d_2+b(X)\}, \\
         \llbracket P(t,X) \rrbracket  \ = \ -\sigma k(\zeta) & \mbox{ on } \Gamma,
         \end{array}
\right.
\end{equation}
where $n$ denotes the unit upward normal vector at the surface at stake.
\bigskip

The next step consists in {\em nondimensionalizing the system}.
Thanks to an appropriate scaling, the two-layer full Euler system~\eqref{eqn:EulerComplet} can be written in dimensionless form.
The study of the linearized system (see~\cite{Lannes13} for example), which can be solved explicitly, leads to a well-adapted rescaling.
\medskip

Let $a$ (resp. $a_b$) be the maximum amplitude of the deformation of the interface. We denote by $\lambda$ a characteristic horizontal length (that we assume to be identical in any of the directions if $d=2$; see~\cite{Lannes} for a treatment of the anisotropic case when $d=2$), say the wavelength of the interface. Then the typical velocity of small propagating internal waves (or wave celerity) is given by
\[c_0 \ = \ \sqrt{g\frac{(\rho_2-\rho_1) d_1 d_2}{\rho_2 d_1+\rho_1 d_2}}.\]
Consequently, we introduce the dimensionless variables\footnote{We choose $d_1$ as the reference vertical length. This choice is harmless as we assume in the following that the two layers of fluid have comparable depth: the depth ratio $\delta$ does not approach zero or infinity. 
}
\[\begin{array}{cccc}
 \t z \ \equiv\  \dfrac{z}{d_1}, \quad\quad & \t X\ \equiv \ \dfrac{X}{\lambda}, \quad\quad & \t t\ \equiv\ \dfrac{c_0}{\lambda}t,
\end{array}\]
the dimensionless unknowns
\[\begin{array}{cc}
 \t{\zeta}(\t t,\t X)\ \equiv\ \dfrac{\zeta(t,X)}{a},  \quad  \t{b}(\t X)\ \equiv\ \dfrac{b(X)}{a_b}, \quad\quad& \t{\phi_i}(\t t,\t X,\t z)\ \equiv\ \dfrac{d_1}{a\lambda c_0}\phi_i(t,X,z) \quad (i=1,2),
\end{array}\]
as well as the following dimensionless parameters
\begin{equation}\label{parameters}
 \gamma\ =\ \dfrac{\rho_1}{\rho_2}, \quad \epsilon\ \equiv\ \dfrac{a}{d_1}, \quad\beta\ \equiv\ \dfrac{a_b}{d_1},\quad   \mu\ \equiv\ \dfrac{d_1^2}{\lambda^2},  \quad \delta\ \equiv \ \dfrac{d_1}{d_2}, \quad \Bo\ =\ \dfrac{g(\rho_2-\rho_1)\lambda^2}{\sigma}.
\end{equation}
\bigskip

We conclude by remarking that the system can be reduced into two evolution equations coupling Zakharov's canonical variables~\cite{Zakharov68,CraigSulem93}, namely  (withdrawing the tildes for the sake of readability)
the deformation of the free interface from its rest position, $\zeta$, and the trace of the dimensionless upper potential at the interface, $\psi$, defined as follows:
\[ \psi \ \equiv \ \phi_1(t,X,\zeta(t,X)).\]
Indeed, $\phi_1$ and $\phi_2$ are uniquely deduced from $(\zeta,\psi)$ as solutions of the following Laplace's problems:
\begin{align}
\label{eqn:Laplace1} &\left\{
\begin{array}{ll}
 \left(\ \mu\Delta \ +\  \partial_z^2\ \right)\ \phi_1=0 & \mbox{ in } \Omega_1\equiv \{(X,z)\in \RR^{d+1},\ \epsilon{\zeta}(X)<z<1\}, \\
\partial_n \phi_1 =0  & \mbox{ on } \Gamma_{\text{top}}\equiv \{(X,z)\in \RR^{d+1},\ z=1\}, \\
 \phi_1 =\psi & \mbox{ on } \Gamma\equiv \{(X,z)\in \RR^{d+1},\ z=\epsilon \zeta\},
\end{array}
\right.\\
\label{eqn:Laplace2}&\left\{
\begin{array}{ll}
 \left(\ \mu\Delta\ + \ \partial_z^2\ \right)\ \phi_2=0 & \mbox{ in } \Omega_2\equiv\{(X,z)\in \RR^{d+1},\ -\frac{1}{\delta}+\beta b(X)<z<\epsilon\zeta(X)\}, \\
\partial_{n}\phi_2 = \partial_{n}\phi_1 & \mbox{ on } \Gamma, \\
 \partial_{n}\phi_2 =0 & \mbox{ on } \Gamma_{\text{bot}}\equiv \{(X,z)\in \RR^{d+1},\ z=-\frac{1}{\delta}+\beta b(X)\}.
\end{array}
\right.
\end{align}
More precisely, we define the so-called Dirichlet-Neumann operators.
\begin{Definition}[Dirichlet-Neumann operators]
Let $\zeta,b\in H^{t_0+1}(\RR^d)$, $t_0>d/2$, such that there exists $h_0>0$ with
%\begin{equation} \label{eqn:connected}
$h_1 \ \equiv\  1-\epsilon\zeta  \geq h_0>0$ and $h_2 \ \equiv \ \frac1\delta +\epsilon \zeta-\beta b\geq h_0>0$, and let $\psi\in L^2_{\rm loc}(\RR^d),\nabla \psi\in H^{1/2}(\RR^d)$.
%\end{equation}
 Then we define
  \begin{align*}
 G^{\mu}\psi \ \equiv\  &G^{\mu}[\epsilon\zeta]\psi \ \equiv \ \sqrt{1+\mu|\epsilon\nabla\zeta|^2}\big(\partial_n \phi_1 \big)\id{z=\epsilon\zeta}  \ = \ -\mu\epsilon(\nabla\zeta)\cdot (\nabla\phi_1)\id{z=\epsilon\zeta}+(\partial_z\phi_1)\id{z=\epsilon\zeta},\\
H^{\mu,\delta}\psi \ \equiv\   &H^{\mu,\delta}[\epsilon\zeta,\beta b]\psi\ \equiv\  \big(\phi_2\big)\id{z=\epsilon\zeta} \ = \ \phi_2(t,X,\zeta(t,X)),
\end{align*}
where $\phi_1$ and $\phi_2$ are uniquely defined (up to a constant for $\phi_2$) as the solutions in $H^2(\RR^d)$ of the Laplace's problems~\eqref{eqn:Laplace1}--\eqref{eqn:Laplace2}.
\end{Definition}
The existence and uniqueness of a solution to~\eqref{eqn:Laplace1}--\eqref{eqn:Laplace2}, and therefore the well-posedness of the Dirichlet-Neumann operators follow from classical arguments detailed, for example, in~\cite{Lannes}.
\medskip

Using the above definition, and after straightforward computations, one can rewrite the nondimensionalized version of~\eqref{eqn:EulerComplet} as a simple system of two coupled evolution equations, namely
\begin{equation}\label{eqn:EulerCompletAdim}
\left\{ \begin{array}{l}
\displaystyle\partial_{ t}{\zeta} \ -\ \frac{1}{\mu}G^{\mu}\psi\ =\ 0,  \\ \\
\displaystyle\partial_{ t}\Big(\nabla H^{\mu,\delta}\psi-\gamma \nabla{\psi} \Big)\ + \ (\gamma+\delta)\nabla{\zeta} \ + \ \frac{\epsilon}{2} \nabla\Big(|\nabla H^{\mu,\delta}\psi|^2 -\gamma |\nabla {\psi}|^2 \Big) \\
\displaystyle\hspace{7cm} = \ \mu\epsilon\nabla\N^{\mu,\delta}-\frac{\gamma+\delta}{\Bo}\frac{\nabla \big(k(\epsilon\sqrt\mu\zeta)\big)}{{\epsilon\sqrt\mu}} \ ,
\end{array}
\right.
\end{equation}
where we denote
\[  \N^{\mu,\delta} \ \equiv \ \dfrac{\big(\frac{1}{\mu}G^{\mu}\psi+\epsilon(\nabla{\zeta})\cdot (\nabla H^{\mu,\delta}\psi) \big)^2\ -\ \gamma\big(\frac{1}{\mu}G^{\mu}\psi+\epsilon(\nabla{\zeta})\cdot (\nabla{\psi}) \big)^2}{2(1+\mu|\epsilon\nabla{\zeta}|^2)}.
      \]
We will refer to~\eqref{eqn:EulerCompletAdim} as the {\em full Euler system}, and solutions of this system will be exact solutions of our problem.

\subsection{A well-posedness theorem on the full Euler system}\label{sec:WPfullEuler}
We mention here that Lannes~\cite{Lannes13} recently ensured that the Cauchy problem for~\eqref{eqn:EulerCompletAdim} (with a flat bottom: $\beta=0$) is locally well-posed in Sobolev spaces, with an existence time consistent with observations. Earlier results showed that the problem was ill-posed in the absence of surface tension, outside of the analytic framework. It was subsequently proved that taking into account the surface tension restores the local well-posedness of the equations, but with a very small existence time of the solution when the surface tension is small, which is the case in the oceanographic setting.

It has to be noted that none of the asymptotic models presented in the following sections (and as a matter of fact, no asymptotic model known to us) share the same property, and that the surface tension term could be withdrawn from the equations (by setting $\Bo^{-1}=0$) without hurting their well-posedness. The reason for this apparent paradox is that the positive role of surface tension is to regularize Kelvin-Helmholtz instabilities that appear at high frequencies, while the main part of the wave, which is captured by the asymptotic models, is located at low frequencies and is thus unaffected by surface tension.

In~\cite{Lannes13}, Lannes introduces a {\em stability criterion}, whose role is to ensure that the aforementioned frequency threshold  is high enough, and shows that, under this condition, the combined effect of surface tension and gravity is sufficient to control the regularity of the flow.

Somewhat more precisely, one has the following result; see~\cite[Theorems 5 and 6]{Lannes13} for the precise statements.
 \begin{Theorem} \label{th:WPEuler}
Let $\p\in\P_{SW}$ and the initial data $U^0\equiv (\zeta^0,\psi^0)^\top$ satisfy the following:
\begin{enumerate}
\item $U^0$ belongs to an energy space of sufficiently smooth, bounded functions (in particular, the following is required: $(\zeta^0,\nabla \psi^0)^\top\in H^{9/2}(\RR^d)^{d+1}$);
\item $\zeta^0$ satisfies the non-vanishing depth condition: $\exists h_0>0:\  \min\{ 1-\epsilon\zeta^0 , \delta^{-1}+\epsilon\zeta^0\}\geq h_0$;
\item A stability criterion is satisfied, which can be roughly expressed by $\Upsilon\equiv \epsilon^{-2\varrho}\frac{\rho_1 g a^4}{d_1^2}\frac{1}{4\sigma}\frac{(\gamma+\delta)^2}{(1+\gamma)^6}$ is sufficiently small ($\varrho\in[0,1]$, fixed).
\end{enumerate}
 Then there exists a unique solution to~\eqref{eqn:EulerCompletAdim} (with flat bottom: $\beta=0$) with initial data $U\id{t=0}\equiv U^0$, bounded in the same energy space (no loss of derivatives). The flow is continuous with respect to time, and defined for  $t\in[0,\epsilon^{-\varrho}T]$, where $T>0$ depends only on the quantities defined through the three above conditions, and in particular can be chosen independent of the parameters $\p\in\P_{SW}$.
 \end{Theorem}

\section{The Green-Naghdi/Green-Naghdi model}\label{sec:GNGN}
In the following, we construct Green-Naghdi type models for the system~\eqref{eqn:EulerCompletAdim}, that is asymptotic models with precision $\O(\mu^2)$, in the sense of consistency. As we shall see, and contrarily to earlier works, our construction relies only on asymptotic expansions which can be straightforwardly deduced from known results on the one-layer case. Thus we start by recalling below these results, which can be found in particular in~\cite{Lannes}. We then deduce equivalent asymptotic expansions in the bi-fluidic setting in Section~\ref{sec:Dev2}, and finally use these expansions to construct our asymptotic models in Section~\ref{sec:GN}.

\subsection{Asymptotic expansions in the water-wave case}\label{sec:Dev1}
 The proof of the following statements may be found in~\cite{Lannes} (with depth $D=1$, but the general case is obtained by straightforward change of variables). For simplicity, and without lack of generality, we set $\epsilon=\beta=1$ in this section.

\begin{Definition}[Dirichlet-Neumann operator]\label{defG}
Let $\zeta,b \in H^{t_0+1}(\RR^d)$, $t_0>d/2$, such that there exists $h_0>0$ with
$h \equiv D+\zeta-b \geq h_0>0$, and let $\psi\in L^2_{\rm loc}(\RR^d),\nabla \psi\in H^{1/2}(\RR^d)^d$.
%\end{equation}
 Then we define
\[
\G^{\mu,D}[\zeta,b]\psi \ \equiv \ \sqrt{1+\mu|\nabla\zeta|^2}\big(\partial_n \phi \big)\id{z=\zeta}   \ = \ -\mu(\nabla\zeta)\cdot (\nabla\phi)\id{z=\zeta}+(\partial_z\phi)\id{z=\zeta},
\]
where $\phi\equiv \phi^{\mu,D}[\zeta,b]\psi\in H^2$ is the unique solution to
\begin{equation}
\label{eqn:Laplace0} \left\{
\begin{array}{ll}
 \left(\ \mu\Delta \ +\  \partial_z^2\ \right)\ \phi=0 & \mbox{ in } \Omega\equiv \{(X,z)\in \RR^{d+1},\ -D+b(X)<z<\zeta(X)\}, \\
\partial_n \phi =0  & \mbox{ on } \Gamma_{\text{bot}}\equiv \{(X,z)\in \RR^{d+1},\ z=-D+b(X)\}, \\
 \phi =\psi & \mbox{ on } \Gamma\equiv \{(X,z)\in \RR^{d+1},\ z= \zeta\},
\end{array}\right.
\end{equation}
\end{Definition}

Let us now recall that the Dirichlet-Neumann operator may be equivalently defined through the {\em vertically averaged mean velocity}, thanks to the following Proposition.
\begin{Proposition}\label{prop:idGvsV}
Let $\zeta,b,\psi$ satisfy the assumptions of Definition~\ref{defG}.
Define
  \begin{align*}
\b\V(X) \ \equiv \ \frac1{h(X)}\int_{-D+b(X)}^{\zeta(X)} \nabla\phi(X,z)\ dz
\end{align*}
where $\phi\equiv \phi^{\mu,D}[\zeta,b]\psi \in H^2$ is the unique solution to~\eqref{eqn:Laplace0}.

Then one has the identity
\begin{equation}\label{idGvsV}
\G^{\mu,D}[\zeta,b]\psi \ = \ -\mu \nabla\cdot (\ h\ \b\V\ ).
\end{equation}
\end{Proposition}
\begin{proof}
This striking result is a consequence of a simple calculation, that we recall. Let $\varphi\in C_c^\infty(\RR^d)$ be a test function. Then one has
\begin{align*} 
\int_{\RR^d} \varphi \G^{\mu,D}[\zeta,b]\psi \ dX\ &= \ \int_{\Gamma} \varphi (\partial_n\phi)\ d\Gamma \ = \ \int_{\Omega} (\sqrt{\mu}\nabla \phi) \cdot (\sqrt{\mu}\nabla \varphi) \ d\Omega\\
& = \ \mu \ \int_{\RR^d}dX\ \nabla \varphi \cdot \int_{-D+b}^{\zeta}dz \nabla \phi \\ 
&= \ -\mu \ \int_{\RR^d} \varphi(X)\nabla ( h\ \b\V)\ dX,
\end{align*}
where we used Green's identity, and the Laplace's equation satisfied by~$\phi$. Since this result is valid for any test function $\varphi \in C_c^\infty(\RR^d)$, and as $\G^{\mu,D}[\zeta,b]\psi\in H^{1/2}(\RR^d)$, the identity~\eqref{idGvsV} holds in the strong sense.
\end{proof}

Let us conclude with the asymptotic expansion of the quantities defined above. Here and in the following, we denote, for convenience,\footnote{In order to be completely rigorous, one should take into account the dependence with respect to the parameter $D$ here, and $\delta$ in the subsequent sections. However, this dependence is harmless as we assume that $\delta$ does not approach zero or infinity: $\delta\in[\delta_{\min},\delta_{\max}]$; see~\cite{Lannes13} for example.}
\[ M(s) \ = \ C(h_0^{-1},\big\vert \zeta \big\vert_{H^{\max(s,t_0+1)}},\big\vert b \big\vert_{H^{\max(s,t_0+1)}}) \ .\]
\begin{Proposition}\label{prop:devG}
Let $\zeta,b \in H^{t_0+2}\cap H^{s+4}(\RR^d)$, $t_0>d/2,s\geq0$, such that there exists $h_0>0$ with
$h \equiv D+\zeta-b \geq h_0>0$, and let $\psi\in L^2_{\rm loc}(\RR^d),\nabla \psi\in H^{s+4}(\RR^d)^d$.
Then
\begin{align}
\big\vert \b \V \ - \ \nabla \psi \big\vert_{H^s} \ &\leq \ \mu \ M(s+2) \big\vert \nabla\psi \big\vert_{H^{s+2}}\ ,\\
\big\vert \b \V \ - \ \nabla \psi \ + \ \mu \T[h,b]\nabla\psi \big\vert_{H^s} \ &\leq \ \mu^2\  M(s+4) \big\vert \nabla\psi \big\vert_{H^{s+4}} \ ,
\end{align}
with
\begin{equation}\label{defT}
\T[h,b]V \equiv \frac{-1}{3h}\nabla(h^3\nabla\cdot V)+\frac{1}{2h}\big[\nabla(h^2\nabla b\cdot V)-h^2\nabla b \nabla\cdot V\big]+\nabla b\nabla b\cdot V\ .
\end{equation}

It follows straightforwardly from~\eqref{idGvsV} that if $(\zeta,b ,\nabla \psi)^\top\in H^{s+5}(\RR^d)^{d+2}$, then
\begin{align}
\big\vert \frac1\mu \G^{\mu}[\zeta,b]\psi \ + \ \nabla\cdot (h\nabla \psi) \big\vert_{H^s} \ &\leq \ \mu M(s+3) \big\vert \nabla\psi \big\vert_{H^{s+3}} \ ,\\
\big\vert \frac1\mu \G^{\mu}[\zeta,b]\psi \ + \ \nabla\cdot (h\nabla \psi) \ - \ \mu \nabla\cdot (h\T[h,b]\nabla\psi )\big\vert_{H^s} \ &\leq \ \mu^2 M(s+5) \big\vert \nabla\psi \big\vert_{H^{s+5}} \ .
\end{align}
\end{Proposition}

\subsection{Asymptotic expansions in the bi-fluidic case}\label{sec:Dev2}
Our specific operators may be deduced from the classical Dirichlet-Neumann operator used in the water-wave problem, and Defined in Definition~\ref{defG}. Thus the following results are easily deduced from the ones of the previous section.

Let us first define $\b u_1$ (resp. $\b u_2$) the vertically averaged mean velocity of the upper layer (resp. lower layer):
\begin{Definition}\label{defU1U2}
Let $\zeta,b \in H^{t_0+1}(\RR^d)$, $t_0>d/2$, such that there exists $h_0>0$ with
$h_1 \equiv 1-\epsilon \zeta\geq h_0>0$ and $h_2\equiv \frac1\delta +\epsilon \zeta-\beta b\geq h_0>0$, and let $\psi\in L^2_{\rm loc}(\RR^d),\nabla \psi\in H^{1/2}(\RR^d)^d$. Then we define
\begin{align*}
\b u_1(t,X) \ &= \ \frac{1}{1-\epsilon\zeta(X)}\int_{\epsilon\zeta(X)}^{1} \nabla \phi_1(X,z) \ dz ,\\
  \b u_2(t,X) \ &= \ \frac{1}{\frac1\delta+\epsilon\zeta(X)-\beta b(X)}\int_{-\frac1\delta+\beta  b(X)}^{\epsilon\zeta(X)} \nabla\phi_2(X,z) \ dz.
  \end{align*}
where $\phi_1$ and $\phi_2$ are uniquely defined (up to a constant for $\phi_2$) as the solutions in $H^2(\RR^d)$ of the Laplace's problems~\eqref{eqn:Laplace1}--\eqref{eqn:Laplace2}.
\end{Definition}
\begin{Proposition}\label{prop:idGvsU1U2}
Let $\zeta,b ,\psi$ satisfy the hypothesis of Definition~\ref{defU1U2}.
Then one has the identity
\begin{equation}\label{idGvsU1U2}
G^{\mu}[\epsilon\zeta]\psi \ = \ \mu \nabla\cdot (\ h_1\ \b u_1\ ) \ = \ -\mu \nabla\cdot (\ h_2\ \b u_2\ ).
\end{equation}
\end{Proposition}
The proof of these identities is identical as the one of Proposition~\ref{prop:idGvsV} (when considering the upper and lower potential respectively, and using that $\partial_n\phi_1=\partial_n\phi_2=(1+\mu\vert\epsilon\nabla\zeta\vert^2)^{-1/2}G^{\mu}[\epsilon\zeta]\psi$). Thus we omit the proof, and continue with the asymptotic expansions of the above quantities.

\begin{Proposition}\label{prop:devG1}
Let $\zeta,b \in H^{t_0+2}(\RR^d)\cap H^{s+4}(\RR^d)$, $t_0>d/2,s\geq0$, such that there exists $h_0>0$ with
$\min\{h_1, h_2\}\geq h_0>0$, and let $\psi\in L^2_{\rm loc}(\RR^d),\nabla \psi\in H^{s+4}(\RR^d)^d$.
Then
\begin{align}
\label{u1vsPsi1}\big\vert \b u_1 \ - \ \nabla \psi \big\vert_{H^s} \ &\leq \ \mu M(s+2) \big\vert \nabla\psi \big\vert_{H^{s+2}} \ , \\
\label{u1vsPsi2}\big\vert \b u_1 \ - \ \nabla \psi \ + \ \mu \T[h_1,0]\nabla\psi \big\vert_{H^s} \ &\leq \ \mu^2 M(s+4) \big\vert \nabla\psi \big\vert_{H^{s+4}} \ ,\\
\label{Psivsu1}\big\vert \nabla \psi \ - \ \b u_1 \ - \ \mu \T[h_1,0]\b u_1 \big\vert_{H^s} \ &\leq \ \mu^2 M(s+4) \big\vert \nabla\psi \big\vert_{H^{s+4}} \ .
\end{align}
\end{Proposition}
\begin{proof}
Expansions~\eqref{u1vsPsi1},\eqref{u1vsPsi2},
simply follow from Proposition~\ref{prop:devG} once we remark that $\t\phi(X,z)\equiv\phi_1(X,-z)$ satisfies
\[\left\{
\begin{array}{ll}
 \left(\ \mu\Delta \ +\  \partial_z^2\ \right)\ \t\phi=0 & \mbox{ in }  \{(X,z)\in \RR^{d+1},\ -1<z<-\epsilon\zeta(X)\}, \\
\partial_z \t\phi =0  & \mbox{ on } \{(X,z)\in \RR^{d+1},\ z=-1\}, \\
 \t\phi =\psi& \mbox{ on }  \{(X,z)\in \RR^{d+1},\ z= -\epsilon\zeta\}.
\end{array}\right.
\]
It follows that one has the identity $\phi_1(X,-z)\equiv \phi^{\mu,1}[-\epsilon\zeta,0]\psi$, the unique solution of~\eqref{eqn:Laplace0}. Consequently, $G^{\mu}[\epsilon\zeta]=-\G^{\mu,1}[-\epsilon\zeta,0]$, and the expansions~\eqref{u1vsPsi1},\eqref{u1vsPsi2} follow. Expansion~\eqref{Psivsu1} is a straightforward consequence of~\eqref{u1vsPsi1},\eqref{u1vsPsi2}, and
\[ \big\vert \T[h_1,0](\nabla\psi-\b u_1)\big\vert_{H^s}\leq C(h_0^{-1},\epsilon\big\vert \zeta\big\vert_{H^{s+1}}) \big\vert \nabla\psi-\b u_1\big\vert_{H^{s+2}}. \]
The Proposition is proved.\end{proof}

\begin{Proposition}\label{prop:devH}
Let $\zeta,b\in H^{t_0+2}(\RR^d)\cap H^{s+11/2}(\RR^d)$, $t_0>d/2,s\geq0$, such that there exists $h_0>0$ with
$\min\{h_1, h_2\}\geq h_0>0$, and let $\psi\in L^2_{\rm loc}(\RR^d),\nabla \psi\in H^{s+5}(\RR^d)^d$.
Then one has
\begin{align}
\label{Hpsivsu21} \big\vert \b u_2 \ - \ \nabla H^{\mu,\delta}\psi \big\vert_{H^s} \ &\leq \ \mu M(s+7/2) \big\vert \nabla \psi \big\vert_{H^{s+3}} \ ,\\
\label{Hpsivsu22}  \big\vert  \nabla H^{\mu,\delta}\psi \ - \ \b u_2  \ - \ \mu \T[h_2,\beta b]\b u_2 \big\vert_{H^s} \ &\leq \ \mu^2 M(s+11/2) \big\vert \nabla\psi \big\vert_{H^{s+5}} \ .
\end{align}
\end{Proposition}
\begin{proof}
As above, the expansions can be deduced from Proposition~\ref{prop:devG}, once we remark that by definition, $\phi_2(X,z)$
satisfies
\[\left\{
\begin{array}{ll}
 \left(\ \mu\Delta \ +\  \partial_z^2\ \right)\ \phi_2=0 & \mbox{ in }  \{(X,z)\in \RR^{d+1},\ -\frac1\delta+\beta b(X)<z<\epsilon\zeta(X)\}, \\
\partial_n \phi_2 =0  & \mbox{ on } \{(X,z)\in \RR^{d+1},\ z=-\frac1\delta+\beta b(X)\}, \\
 \phi_2 =H^{\mu,\delta}[\epsilon\zeta,\beta b]\psi& \mbox{ on }  \{(X,z)\in \RR^{d+1},\ z= \epsilon\zeta\}.
\end{array}\right.
\]
In other words, one has the identity $\phi_2( X,z)\equiv \phi^{\mu,\delta^{-1}}[\epsilon\zeta,\beta b]H^{\mu,\delta}\psi$, where $ \phi^{\mu,\delta^{-1}}$ is defined as the solution of~\eqref{eqn:Laplace0}.
% % % % %
\footnote{Note that by definition of the Dirichlet-Neumann operators $G^\mu$ and $\G^{\mu,\delta^{-1}}$, this identity yields
\[ \G^{\mu,\delta^{-1}}[\epsilon\zeta,\beta b]H^{\mu,\delta}\psi \ = \ \sqrt{1+\mu\vert\epsilon\nabla\zeta\vert^2}(\partial_n \phi_2)_{z=\epsilon\zeta} \ = \ G^{\mu}[\epsilon\zeta]\psi  \ .
\]
 In other words, and as remarked in~\cite{Lannes13}, one has the identity
\[ H^{\mu,\delta} \ = \ \big\{\G^{\mu,\delta^{-1}}[\epsilon\zeta,\beta b]\big\}^{-1}G^{\mu}[\epsilon\zeta] \ = \ - \big\{\G^{\mu,\delta^{-1}}[\epsilon\zeta,\beta b]\big\}^{-1}\G^{\mu,1}[-\epsilon\zeta,0]\ . \]
In particular, the bound~\eqref{estHpsi} is not optimal; see~\cite[Proposition 1 and Remark~6]{Lannes13}.}
% % % % % %

Thus one deduces from Proposition~\ref{prop:devG} the following estimates:
\begin{align}
\label{u2vsHpsi1} \big\vert \b u_2 \ - \ \nabla H^{\mu,\delta}\psi \big\vert_{H^s} \ &\leq \ \mu M(s+2) \big\vert \nabla H^{\mu,\delta}\psi \big\vert_{H^{s+2}} \ ,\\
\label{u2vsHpsi2}\big\vert \b u_2 \ - \ \nabla H^{\mu,\delta}\psi \ + \ \mu \T[h_2,\beta b]\nabla H^{\mu,\delta}\psi \big\vert_{H^s} \ &\leq \ \mu^2 M(s+4) \big\vert \nabla H^{\mu,\delta}\psi \big\vert_{H^{s+4}}\ .
\end{align}
Furthermore, one has from~\cite[Proposition 3]{BonaLannesSaut08} that 
\begin{equation}\label{estHpsi}
\big\vert  \nabla H^{\mu,\delta}\psi \big\vert_{H^s} \leq C(h_0^{-1},\delta,\delta^{-1},|\zeta|_{H^{s+3/2}})| \nabla \psi |_{H^{s+1}},
\end{equation}
so that estimate~\eqref{Hpsivsu21} is now straightforward. 

Finally, estimate~\eqref{Hpsivsu22} is easily deduced from the previous estimates.
\end{proof}

\subsection{Construction of the Green-Naghdi/Green-Naghdi model}\label{sec:GN}
Let us recall the full Euler system~\eqref{eqn:EulerCompletAdim}:
\begin{equation}\label{eqn:EulerCompletAdim2}
\left\{ \begin{array}{l}
\displaystyle\partial_{ t}{\zeta} \ -\ \frac{1}{\mu}G^{\mu}\psi\ =\ 0,  \\ \\
\displaystyle\partial_{ t}\Big(\nabla H^{\mu,\delta}\psi-\gamma \nabla{\psi} \Big)\ + \ (\gamma+\delta)\nabla{\zeta} \ + \ \frac{\epsilon}{2} \nabla\Big(|\nabla H^{\mu,\delta}\psi|^2 -\gamma |\nabla {\psi}|^2 \Big) \\
\displaystyle\hspace{7cm} = \ \mu\epsilon\nabla\N^{\mu,\delta}-\frac{\gamma+\delta}{\Bo}\frac{\nabla \big(k(\epsilon\sqrt\mu\zeta)\big)}{{\epsilon\sqrt\mu}} \ ,
\end{array}
\right.
\end{equation}
where we denote
$ \N^{\mu,\delta} \ \equiv \ \dfrac{\big(\frac{1}{\mu}G^{\mu}\psi+\epsilon(\nabla{\zeta})\cdot (\nabla H^{\mu,\delta}\psi) \big)^2\ -\ \gamma\big(\frac{1}{\mu}G^{\mu}\psi+\epsilon(\nabla{\zeta})\cdot (\nabla{\psi}) \big)^2}{2(1+\mu|\epsilon\nabla{\zeta}|^2)}.
     $

By Proposition~\ref{prop:idGvsU1U2}, the first equation of~\eqref{eqn:EulerCompletAdim2} yields
\begin{equation}\label{dtzetavsu1u2}
\partial_{ t}{\zeta} \ = \ \nabla\cdot(h_1\b u_1)\ =\ -\nabla\cdot(h_2\b u_2) .\end{equation}
When we plug the expansions of Propositions~\ref{prop:devG1} and~\ref{prop:devH} into the second equation of~\eqref{eqn:EulerCompletAdim2}, and withdrawing $\O(\mu^2)$ terms, one obtains
\begin{multline}\label{eqn:GNGN}
\partial_{ t}\Big(\b u_2-\gamma \b u_1 +  \mu \T[h_2,\beta b]\b u_2-  \mu \gamma \T[h_1,0]\b u_1\Big)\ \\
+ \ (\gamma+\delta)\nabla{\zeta} \ + \ \frac{\epsilon}{2} \nabla\Big(|\b u_2  \ + \ \mu \T[h_2,\beta b]\b u_2|^2 -\gamma |\b u_1  \ + \ \mu \T[h_1,0]\b u_1|^2 \Big) \\
= \ \mu\epsilon\nabla\N^{\mu,\delta}_0-\frac{\gamma+\delta}{\Bo}\frac{\nabla \big(k(\epsilon\sqrt\mu\zeta)\big)}{{\epsilon\sqrt\mu}} \ + \ \O(\mu^2)\ ,
\end{multline}
with
\begin{align*} \N^{\mu,\delta}_0 \ &\equiv \ \dfrac{\big(-\nabla\cdot(h_2\b u_2)+\epsilon(\nabla{\zeta})\cdot (\b u_2) \big)^2\ -\ \gamma\big(\nabla\cdot(h_1\b u_1)+\epsilon(\nabla{\zeta})\cdot (\b u_1) \big)^2}{2}\\
&\equiv \ \dfrac{\big(-h_2\nabla\b u_2+\beta(\nabla{b})\cdot (\b u_2) \big)^2\ -\ \gamma\big(h_1\nabla\cdot\b u_1 \big)^2}{2}
\end{align*}

\begin{Remark}
Equations~\eqref{dtzetavsu1u2} and~\eqref{eqn:GNGN} are very similar to the system obtained in~\cite{ChoiCamassa99}. It may also be recovered from system (60) in~\cite{Duchene10} when setting $\alpha=0$ (notation therein), and after straightforward adjustments (in particular, we use a different scaling in the non-dimensionalizing step).
\end{Remark}
\begin{Proposition}[Consistency] \label{cons:GNu1u2}
Let $U^\p\equiv(\zeta^\p,\psi^\p)^\top$ be a family of solutions to the full Euler system~\eqref{eqn:EulerCompletAdim} such that there exists $T>0$, $s\geq0$ for which $(\zeta^\p,\nabla\psi^\p)^\top$ is bounded in $L^\infty([0,T);H^{s+N})^{d+1}$ ($N$ sufficiently large), uniformly with respect to $(\mu,\epsilon,\delta,\gamma)\equiv \p\in \P_{SW}$; see Definition~\ref{Def:Regimes}. Moreover, assume that $b\in  H^{s+N}$ and
\[
\exists h_{0}>0 \text{ such that }\quad h_1 \ \equiv\  1-\epsilon\zeta^\p \geq\ h_{0}\ >\ 0, \quad h_2 \ \equiv \ \frac1\delta +\epsilon \zeta^\p-\beta b\geq\ h_{0}\ >\ 0.
\]
Define $\b u_1^\p,\b u_2^\p$ as in Definition~\ref{defU1U2}. Then $(\zeta^\p,\b u_1^\p,\b u_2^\p)$  satisfy~\eqref{dtzetavsu1u2},\eqref{eqn:GNGN}, the latter up to a remainder term, $R$, bounded by
\[ \big\Vert R \big\Vert_{L^\infty([0,T);H^s)^d} \ \leq \ \mu^2\ C \ , \]
with $C=C(h_0^{-1},\mu_{\max},\delta_{\min}^{-1},\delta_{\max},\big\vert b \big\vert_{H^{s+N}} ,\big\Vert \zeta^\p\big\Vert_{L^\infty([0,T);H^{s+N})} ,\big\Vert\nabla\psi^\p\big\Vert_{L^\infty([0,T);H^{s+N})^d}$, uniform with respect to $(\mu,\epsilon,\delta,\gamma)\in \P_{SW}$.
\end{Proposition}
\begin{proof}
The fact that $(\zeta^\p,\b u_1^\p,\b u_2^\p)$ satisfy~\eqref{dtzetavsu1u2} has been expressed earlier in Proposition~\ref{prop:idGvsU1U2}. That~\eqref{eqn:GNGN} approximately holds is a consequence of the asymptotic expansions of Propositions~\ref{prop:devG1} and~\ref{prop:devH}. Let us detail briefly the argument.

Subtracting~\eqref{eqn:GNGN} to the last equation in~\eqref{eqn:EulerCompletAdim} yields (withdrawing the explicit reference to $\p\in \P_{SW}$ for the sake of readability)
\begin{align} R \ &= \  \partial_{ t}\Big(\nabla H^{\mu,\delta}\psi-\gamma \nabla{\psi} -\big\{\b u_2-\gamma \b u_1 +  \mu \T[h_2,\beta b]\b u_2-  \mu \gamma \T[h_1,0]\b u_1\big\}\Big)\nn \\
& \quad + \ \frac{\epsilon}{2} \nabla\Big(|\nabla H^{\mu,\delta}\psi|^2 -\gamma |\nabla {\psi}|^2- \big\{ |\b u_2  \ + \ \mu \T[h_2,\beta b]\b u_2|^2 -\gamma |\b u_1  \ + \ \mu \T[h_1,0]\b u_1|^2  \big\} \Big) \nn \\
 &\quad- \ \mu\epsilon\nabla \big(\N^{\mu,\delta}-\N^{\mu,\delta}_0\big) \nn\\
 &\equiv \ R_I \ +  \ R_{II} \ + \ R_{III}.
\end{align}
We now show how to estimate each of these terms.
\medskip

Recall~\eqref{Hpsivsu22} in Proposition~\ref{prop:devH}. It follows from tame product estimates in $H^s$ (see~\cite[Appendix~A]{DucheneIsrawiTalhouk} for example)
\begin{align*}
\big\vert |\nabla H^{\mu,\delta}\psi|^2-  |\b u_2  \ + \ \mu \T[h_2,\beta b]\b u_2|^2 \big\vert_{H^s} \ &\leq \ C_2 \big\vert \nabla H^{\mu,\delta}\psi\ -\ \b u_2  \ - \ \mu \T[h_2,\beta b]\b u_2\big\vert_{H^s} \\
&\leq\ \mu^2 C_2 \ M(s+11/2) \big\vert \nabla \psi \big\vert_{H^{s+5}}
\end{align*}
with $C_2\lesssim \big\vert \nabla H^{\mu,\delta}\psi + \b u_2   +  \mu \T[h_2,\beta b]\b u_2\big\vert_{H^{\max\{t_0,s\}}} $. Here and below, we denote $t_0>d/2$.

Identically, using~\eqref{Psivsu1} in Proposition~\ref{prop:devG1}, one obtains
\begin{align*}
\big\vert |\nabla {\psi}|^2-  |\b u_1  \ + \ \mu \T[h_1,0]\b u_1|^2  \big\vert_{H^s} \ &\leq \ C_1 \big\vert \nabla {\psi}\ -\   \b u_1  \ - \ \mu \T[h_1,0]\b u_1 \big\vert_{H^s} \\
&\leq\ \mu^2 C_1 \ M(s+4) \big\vert \nabla \psi \big\vert_{H^{s+4}}
\end{align*}
with $C_1\lesssim \big\vert  \nabla {\psi} +   \b u_1  + \mu \T[h_1,0]\b u_1 \big\vert_{H^{\max\{t_0,s\}}} $.

It is now clear that one can choose $N$ sufficiently large so that the following holds:
\begin{equation}\label{est:RII}
\big\Vert R_{II} \big\Vert_{L^\infty([0,T);H^s)^d} \ \leq \ \mu^2 C \ ,
\end{equation}
with $C$ as in the Proposition  (note that one can deduce a control in $H^s$ of $\b u_1$ from a control in $H^{s+2}$ of $\nabla\psi$, thanks to~\eqref{u1vsPsi1} ---being non optimal).
\medskip

The estimate on $R_{III}$ is obtained similarly. Using identity~\eqref{idGvsU1U2} as well as first order expansions~\eqref{u1vsPsi1},\eqref{u2vsHpsi1}, one obtains
\[ \big\vert \big(\frac{1}{\mu}G^{\mu}\psi+\epsilon(\nabla{\zeta})\cdot (\nabla H^{\mu,\delta}\psi) \big)^2-\big(-\nabla\cdot(h_2\b u_2)+\epsilon(\nabla{\zeta})\cdot (\b u_2) \big)^2 \big\vert_{H^s} \ \leq \ \mu \ C_2'\ M(s+7/2)\big\vert \nabla \psi\big\vert_{H^{s+3}},\]
with $C_2'\lesssim \big\vert \frac{2}{\mu}G^{\mu}\psi + \epsilon(\nabla{\zeta})\cdot (\nabla H^{\mu,\delta}\psi+ \b u_2)\big\vert_{H^{\max\{t_0,s\}}} $, and
\[ \big\vert\big(\frac{1}{\mu}G^{\mu}\psi+\epsilon(\nabla{\zeta})\cdot (\nabla{\psi}) \big)^2-\big(\nabla\cdot(h_1\b u_1)+\epsilon(\nabla{\zeta})\cdot (\b u_1) \big)^2 \big\vert_{H^s} \ \leq \ \mu \ C_1'\ M(s+2)\big\vert \nabla \psi\big\vert_{H^{s+2}},\]
with $C_1'\lesssim \big\vert \frac{2}{\mu}G^{\mu}\psi + \epsilon(\nabla{\zeta})\cdot (\nabla\psi+ \b u_1)\big\vert_{H^{\max\{t_0,s\}}} $.

Finally, for any $f\in H^s(\RR^d)$, one has
\[ \big\vert \frac{f}{2(1+\mu|\epsilon\nabla{\zeta}|^2)} - \frac{f}{2}  \big\vert_{H^s} \ \lesssim \ \big\vert f  \big\vert_{H^s}  \big\vert \mu|\epsilon\nabla{\zeta}|^2\big\vert_{H^{\max\{t_0,s\}}}, \]
so that one deduces from the above estimates that
\begin{equation}\label{est:RIII}
\big\Vert R_{III} \big\Vert_{L^\infty([0,T);H^s)^d} \ \leq \ \mu^2 C \ ,
\end{equation}
with $C$ as in the Proposition, and for $N$ sufficiently large.
\medskip

The estimate on $R_I$ requires a control of the time derivatives. One can obtain equivalent results as in Propositions~\ref{prop:devG1} an~\ref{prop:devH}, and in particular
\begin{align*}
\big\vert \partial_t\Big( \nabla \psi \ - \ \b u_1 \ - \ \mu \T[h_1,0]\b u_1 \Big) \big\vert_{H^s} \ &\leq \ \mu^2 N(s+4) \big\vert \partial_t \nabla\psi \big\vert_{H^{s+4}}\\
  \big\vert \partial_t\Big( \nabla H^{\mu,\delta}\psi \ - \ \b u_2  \ - \ \mu \T[h_2,\beta b]\b u_2\Big) \big\vert_{H^s} \ &\leq \ \mu^2  N(s+11/2)\big\vert \partial_t \nabla\psi \big\vert_{H^{s+5}},
\end{align*}
with $N(\b s)\equiv C(\frac{1}{h_0},\big\vert \zeta \big\vert_{H^{\b s}},\big\vert b \big\vert_{H^{\b s}},\big\vert  \nabla\psi \big\vert_{H^{\b s}},\big\vert \partial_t \zeta \big\vert_{H^{\b s}})$ for $\b s\geq t_0+1$;
following the same method, but after differentiating the equations (with respect to the time variable, $t$). We do not detail the proof, and refer to~\cite{Duchene10,Lannes} for examples of applications of this strategy.

Finally, note that one can control $\partial_t \nabla\psi $ and $\partial_t \zeta$ using only a control on $\nabla\psi$ and $\zeta$, using that $(\zeta,\psi)^\top$ satisfies the full Euler system~\eqref{eqn:EulerCompletAdim}; allowing for a loss of derivatives. At the end of the day, one sees that if $N$ is sufficiently large, then one has
\begin{equation}\label{est:RI}
\big\Vert R_{I} \big\Vert_{L^\infty([0,T);H^s)^d} \ \leq \ \mu^2 C \ ,
\end{equation}
with $C$ as in the Proposition. 
\medskip

Altogether, estimates~\eqref{est:RII},~\eqref{est:RIII} and~\eqref{est:RI} yield Proposition~\ref{cons:GNu1u2}.
\end{proof}
\bigskip

Our aim is now to approximate~\eqref{dtzetavsu1u2},\eqref{eqn:GNGN} as a system of coupled evolution equations (thus directly comparable with~\eqref{eqn:EulerCompletAdim2}). In order to do so, we introduce a new velocity variable, $v$, which shall satisfy
\begin{equation}\label{eqnv}
\nabla\cdot (\frac{h_1h_2}{h_1+\gamma h_2}v) \ = \ \nabla\cdot (h_2\b u_2) \ = \ -\nabla\cdot(h_1\b u_1),\end{equation}
so that
\begin{equation}\label{dtzetavsv}
\partial_{ t}{\zeta} \ + \ \nabla\cdot (\frac{h_1h_2}{h_1+\gamma h_2}v) \ = \ 0 \ .\end{equation}

In dimension $d=1$, identity~\eqref{eqnv} (assuming that $v\to 0$ as $|x|\to\infty$) uniquely defines $v$ as the {\em shear mean velocity}:
 \begin{equation}\label{eqn:defbarv}
\partial_x (\frac{h_1h_2}{h_1+\gamma h_2}v) \ = \ \partial_x (h_2\b u_2) \ = \ -\partial_x(h_1\b u_1) \ \ \text{ if and only if }\ \ v \ = \ \b u_2 \ - \ \gamma \b u_1 \ .
\end{equation}

However, such an explicit expression is not available in dimension $d=2$. In that case, we make use of the following operator, defined in~\cite{BonaLannesSaut08}.
\begin{Lemma}\label{LemQ}
Assume that $\xi\in L^\infty(\RR^d)$ be such that $\vert \xi\vert_{L^\infty}<1$. Then for any $W\in L^2(\RR^d)^d$, there exists a unique $V\in  L^2(\RR^d)^d$ such that 
\[ \nabla\cdot \big( (1+\xi) V\big) \ = \ \nabla\cdot W\ .\]
and $\Pi V  =  V$, where $\Pi=\frac{\nabla\nabla^\top}{|D|^2}$ is the orthogonal projector onto the gradient vector fields of $L^2(\RR^d)^d$.

Moreover, one has $V=\mfQ[\xi]W$, where $\mfQ[\xi]:L^2(\RR^d)^d\to L^2(\RR^d)^d$ is defined by
\[ \mfQ[\xi]:U\mapsto \sum_{n=0}^\infty (-1)^n\big(\Pi(\xi\Pi\cdot)\big)^n (\Pi U).\]

Furthermore, if $\xi\in H^{s}(\RR^d)$ and $W\in H^s(\RR^d)^d$ with $s\geq t_0+1,t_0>d/2$, then $\mfQ[\xi]W\in H^s(\RR^d)^d$ and
\[\big\vert \mfQ[\xi]W\big\vert_{H^s} \ \leq \ C\Big(\vert \xi\vert_{H^s},\frac{1}{1-\vert\xi\vert_{L^\infty}}\Big)\big\vert W\big\vert_{H^s}.\]
\end{Lemma}

This allows to define $v$ as the unique gradient solution to~\eqref{eqnv}.
\begin{Definition} \label{defV} Let $\zeta\in L^\infty(\RR^d)$ be such that $\epsilon\vert \zeta\vert_{L^\infty}<1$ and $\vert \epsilon\zeta-\beta b \vert_{L^\infty}<\delta^{-1}$, so that 
\[ \frac{h_1h_2}{h_1+\gamma h_2} \ = \ \frac{1}{\gamma+\delta}(1+\xi), \quad \text{with} \quad |\xi|_{L^\infty}<1.\]
 Then we define $v$ as the unique gradient solution to~\eqref{eqnv}; or, in other words,
 \[ v \ = \ -(\gamma+\delta)\mfQ[\xi](h_1 \b u_1) \ = \  (\gamma+\delta)\mfQ[\xi](h_2 \b u_2) \ .\]
\end{Definition}
Note that the condition $ |\xi|_{L^\infty}<1$ is ensured by the following:
\[ \xi \ = \ (\gamma+\delta)\frac{h_1h_2}{h_1+\gamma h_2}-1 \ = \ \frac{h_1(\delta h_2-1)+\gamma h_2 (h_1-1)}{h_1+\gamma h_2} \ = \ 1-\frac{h_1(2-\delta h_2)+\gamma h_2 (2-h_1)}{h_1+\gamma h_2} ,\]
and $\epsilon\vert \zeta\vert_{L^\infty}<1$ yields  $0\leq h_1  \leq 2$ whereas $\vert \epsilon\zeta-\beta b \vert_{L^\infty}<\delta^{-1}$ yields $0\leq h_2  \leq 2\delta^{-1}$.
\begin{Remark}\label{remd1}
In the one dimensional space ($d=1$), one has $\Pi=Id$, and one can check that the operator $\mfQ$ simply reduces to $\mfQ[\xi]W=\frac{1}{1+\xi}W$, so that one recovers 
\[ v \ = \ \b u_2-\gamma\b u_1 \quad \text{ and } \quad \b u_1 \ = \ \frac{-h_2v}{h_1+\gamma h_2} \ , \ \b u_2 \ = \ \frac{h_1 v}{h_1+\gamma h_2} \ .\] 
Note also that in that case, conditions $\epsilon\vert \zeta\vert_{L^\infty}<1$ and $\vert \epsilon\zeta-\beta b \vert_{L^\infty}<\delta^{-1}$ may be replaced by the slightly less stringent non-vanishing depth condition: $h_1=1-\epsilon\zeta>0,h_2=\delta^{-1}+\epsilon\zeta-\beta b>0$.
\end{Remark}

Let us emphasize again that in the case $d=2$, there is no reason to think that $v=\b u_2-\gamma\b u_1$ holds, even with precision $\O(\mu)$; and in particular that $\b u_2-\gamma\b u_1$ is a gradient vector fields. In the same way,
one would like to write, seeing~\eqref{eqnv}, $\b u_2=\delta \mfQ[\epsilon\delta\zeta] \big(\frac{h_1h_2}{h_1+\gamma h_2}v\big)$ and $\b u_1=-\mfQ[-\epsilon\zeta]\big( \frac{h_1h_2}{h_1+\gamma h_2}v\big)$; but unfortunately, it is not clear that $\b u_1,\b u_2$ are gradient vector fields (as a matter of fact, their second order expansion tends to show that it is not true). However, one has the following expansion:
\begin{Proposition}\label{prop:devu1u2v}
Let $s\geq t_0+1$, $t_0>d/2$, $\psi\in L^2_{\rm loc}(\RR^d),\nabla \psi\in H^{s+11/2}(\RR^d)^d$ and $\zeta,b\in  H^{s+5}(\RR^d)$ be such that 
\[ \exists h_{0}>0 \text{ such that }\quad   1-\epsilon \big\vert \zeta\big\vert_{L^\infty} \geq\ h_{0}\ >\ 0, \quad  \frac1\delta -\big\vert \beta b-\epsilon\zeta \big\vert_{L^\infty} \geq\ h_{0}\ >\ 0. \]
 Then one has
\begin{align}
\label{u1vsv1} \big\vert \nabla \psi \ - \ \t u_1 \big\vert_{H^s}  \ +  \ \big\vert \b u_1 \ - \ \t u_1 \big\vert_{H^s}  &\leq \mu  M(s+2)\big\vert \nabla \psi \big\vert_{H^{s+2}}, \\
\label{u2vsv1}\big\vert \nabla H^{\mu,\delta}\psi \ -\ \t u_2 \big\vert_{H^s} \ + \ \big\vert \b u_2\ -\ \t u_2 \big\vert_{H^s}  &\leq \mu  M(s+7/2)\big\vert \nabla \psi \big\vert_{H^{s+3}}, \\
\label{u1vsv2}\big\vert \nabla\psi\ -\ \t u_1\ -\ \mu \mfQ[-\epsilon\zeta]( h_1 \T[h_1,0] \t u_1) \big\vert_{H^s}  &\leq \mu^2  M(s+4)\big\vert \nabla \psi \big\vert_{H^{s+4}} ,\\
\label{u2vsv2}\big\vert \nabla H^{\mu,\delta}\psi \ -\  \t u_2\  -\ \mu\delta \mfQ[\delta\epsilon\zeta]( h_2 \T[h_2,\beta b] \t u_2)\big\vert_{H^s} &   \leq \mu^2  M(s+11/2)\big\vert \nabla \psi \big\vert_{H^{s+5}} ,
\end{align}
where we denote $\t u_1\equiv -\mfQ[-\epsilon\zeta] \big(\frac{h_1h_2}{h_1+\gamma h_2}v\big)$, $\t u_2\equiv \delta \mfQ[\delta\epsilon\zeta] \big(\frac{h_1h_2}{h_1+\gamma h_2}v\big)$.
\end{Proposition}
\begin{proof}
The first estimate follows from 
\[ \nabla\cdot(h_1\nabla\psi) \ = \ \nabla\cdot (h_1\b u_1)+\nabla\cdot \big(h_1 (\nabla\psi-\b u_1) \big) \ = \ - \nabla\cdot \big( \frac{h_1h_2}{h_1+\gamma h_2}v\big)+\nabla\cdot \big(h_1 (\nabla\psi-\b u_1) \big), \]
where we used identity~\eqref{eqnv}. Consequently, Lemma~\ref{LemQ} yields
\[ \nabla\psi \ = \ -\mfQ[-\epsilon\zeta] \big( \frac{h_1h_2}{h_1+\gamma h_2}v-h_1 (\nabla\psi-\b u_1) \big) \ .\]
The control of $\big\vert \nabla \psi \ - \ \t u_1 \big\vert_{H^s}$ as in estimate~\eqref{u1vsv1} is now a consequence of Proposition~\ref{prop:devG1}, and the continuity of the operator $\mfQ$ expressed in Lemma~\ref{LemQ}. The control of $\big\vert \nabla \psi \ - \ \t u_1 \big\vert_{H^s}$ is then deduced, using once again Proposition~\ref{prop:devG1} and triangular inequality. Estimate~\eqref{u1vsv1} is proved.

 Estimate~\eqref{u2vsv1} is obtained in the same way, but using the control of $\big\vert \b u_2-\nabla H^{\mu,\delta}\psi\big\vert_{H^s}$ displayed in Proposition~\ref{prop:devH}.

Following the same strategy, one order further, yields
\begin{align*} \nabla\cdot(h_1\nabla\psi) \ &= \ \nabla\cdot (h_1\b u_1)+\mu \nabla\cdot{h_1 \T[h_1,0]\b u_1}+\nabla\cdot \big(h_1 (\nabla\psi-\b u_1-\mu\T[h_1,0]\b u_1) \big) \\
&= \ -\nabla\cdot\big( \frac{h_1h_2}{h_1+\gamma h_2}v\big)-\mu \nabla\cdot\big(  h_1 \T[h_1,0] \mfQ[-\epsilon\zeta] \big(\frac{h_1h_2}{h_1+\gamma h_2}v\big)\big) \\
& \qquad + \nabla\cdot\Big(h_1 (\nabla\psi-\b u_1-\mu\T[h_1,0]\b u_1) +\mu h_1 \T[h_1,0]\big\{\b u_1+\mfQ[-\epsilon\zeta] \big(\frac{h_1h_2}{h_1+\gamma h_2}v\big) \big\}\Big).
\end{align*}
The last term in the identity above is estimated in part thanks to Proposition~\ref{prop:devG1}, and in part thanks to the first order estimate~\eqref{u1vsv1}. Estimate~\eqref{u1vsv2} then follows as above from the definition and continuity of the the operator $\mfQ$; see Lemma~\ref{LemQ}.

Estimate~\eqref{u2vsv2} is obtained in the same way, and we omit the detailed calculations.
\end{proof}

One deduces from Proposition~\ref{prop:devu1u2v} the following approximate equation equivalent to~\eqref{eqn:GNGN} (withdrawing again $\O(\mu^2)$ terms):
\begin{multline}\label{eqn:GNGNv}
\partial_{ t}\Big(\t u_2 -\gamma \t u_1 +\mu \big(\delta \mfQ[\delta\epsilon\zeta]( h_2 \T[h_2,\beta b] \t u_2)-   \gamma \mfQ[-\epsilon\zeta]( h_1 \T[h_1,0] \t u_1 )\big)\Big)\ + \ (\gamma+\delta)\nabla{\zeta}\\
 \ + \ \frac{\epsilon}{2} \nabla\Big(|\big(I+\mu\delta\mfQ[\delta\epsilon\zeta]h_2\T[h_2,\beta b]\big)\t u_2 |^2 -\gamma |\big(I+\mu \mfQ[-\epsilon\zeta]h_1\T[h_1,0]\big)\t u_1|^2 \Big) \\
= \ \mu\epsilon\nabla\t\N^{\mu,\delta}_0-\frac{\gamma+\delta}{\Bo}\frac{\nabla \big(k(\epsilon\sqrt\mu\zeta)\big)}{{\epsilon\sqrt\mu}} \ + \ \O(\mu^2)\ ,
\end{multline}
where we denote $\t u_1\equiv -\mfQ[-\epsilon\zeta] \big(\frac{h_1h_2}{h_1+\gamma h_2}v\big)$, $\t u_2\equiv \delta \mfQ[\epsilon\delta\zeta] \big(\frac{h_1h_2}{h_1+\gamma h_2}v\big)$, and
\[\t\N^{\mu,\delta}_0 \ \equiv \ \dfrac{\big(-h_2\nabla\cdot \t u_2+\beta(\nabla{b})\cdot (\t u_2) \big)^2\ -\ \gamma\big(h_1\nabla\cdot\t u_1 \big)^2}{2}.
\]

\begin{Proposition}[Consistency] \label{cons:GNv}
Let $U^\p\equiv(\zeta^\p,\psi^\p)^\top$ be a family of solutions to the full Euler system~\eqref{eqn:EulerCompletAdim} such that there exists $T>0$, $s\geq0$ for which $(\zeta^\p,\nabla\psi^\p)^\top$ is bounded in $L^\infty([0,T);H^{s+N})^{d+1}$ ($N$ sufficiently large), uniformly with respect to $(\mu,\epsilon,\delta,\gamma)\equiv \p\in \P_{SW}$; see Definition~\ref{Def:Regimes}. Moreover, assume that $b\in H^{s+N}$ and
\[ \exists h_{0}>0 \text{ such that }\quad   1-\epsilon \big\vert \zeta\big\vert_{L^\infty} \geq\ h_{0}\ >\ 0, \quad  \frac1\delta -\big\vert \beta b-\epsilon\zeta \big\vert_{L^\infty} \geq\ h_{0}\ >\ 0. \]
Define $v^\p$ through Definitions~\ref{defU1U2} and~\ref{defV}. Then $(\zeta^\p,v^\p)^\top$  satisfy~\eqref{dtzetavsv} and~\eqref{eqn:GNGNv}, the latter up to a remainder term, $R$, bounded by
\[ \big\Vert R \big\Vert_{L^\infty([0,T);H^s)^d} \ \leq \ \mu^2\ C \ ,\]
with $C=C(h_0^{-1},\mu_{\max},\delta_{\min}^{-1},\delta_{\max},\big\vert b \big\vert_{H^{s+n}},\big\Vert \zeta^\p\big\Vert_{L^\infty([0,T);H^{s+N})} ,\big\Vert\nabla\psi^\p\big\Vert_{L^\infty([0,T);H^{s+N})^d})$, uniform with respect to $(\mu,\epsilon,\delta,\gamma)\in \P_{SW}$.
\end{Proposition}
The Proposition is obtained as in the proof of Proposition~\ref{cons:GNu1u2}, but using the asymptotic expansions of Proposition~\ref{prop:devu1u2v}; we omit its proof.

%\begin{Remark}[Case $d=1$]
\paragraph{Unidimensional case ($d=1$).} Recall that in the one dimensional space, one has simply
\[(\gamma+\delta)\mfQ[\xi]V \ \equiv \ \frac{1}{h_1+\gamma h_2}V \ , \qquad  \mfQ[-\epsilon\zeta]V \ \equiv \ \frac{1}{h_1}V \quad \text{ and } \quad  \delta\mfQ[\delta\epsilon\zeta]V \ \equiv \ \frac{1}{h_2}V \ ,\]
for any $V\in L^2(\RR)$, and denoting $h_1\equiv 1-\epsilon \zeta$ and $h_2\equiv\delta^{-1}+\epsilon\zeta$.
 In particular, one can check that $\t u_1=\b u_1=\frac{-h_2\ v}{h_1+\gamma h_2}$ and $\t u_2=\b u_2=\frac{h_1\ v}{h_1+\gamma h_2}$.
The system~\eqref{dtzetavsv},\eqref{eqn:GNGNv} thus becomes
\begin{equation}\label{eqn:GNGNd1}
\left\{ \begin{array}{l}
\dsp \partial_{ t}{\zeta} \ + \ \partial_x (\frac{h_1h_2\ v}{h_1+\gamma h_2}) \ = \ 0 \ ,\\ \\
\partial_{ t}\Big(v +  \mu \T[h_2,\beta b]\big(\frac{h_1\ v}{h_1+\gamma h_2}\big)-  \mu \gamma \T[h_1,0]\big(\frac{-h_2\ v}{h_1+\gamma h_2}\big)\Big)\ + \ (\gamma+\delta)\partial_x{\zeta}\\
\qquad\qquad \ + \ \frac{\epsilon}{2} \partial_x \Big(\frac{(h_1)^2-\gamma (h_2)^2}{(h_1+\gamma h_2)^2}v^2\Big) \ = \ \mu\epsilon\partial_x\R^{\mu,\delta}_0-\frac{\gamma+\delta}{\Bo}\frac{\partial_x \big(k(\epsilon\sqrt\mu\zeta)\big)}{{\epsilon\sqrt\mu}} \ + \ \O(\mu^2)\ ,
\end{array}\right.
\end{equation}
where we denote 
\begin{multline*}\R^{\mu,\delta}_0 \ \equiv \ \frac12\big(-h_2\partial_x\big(\frac{h_1\ v}{h_1+\gamma h_2}\big)+\beta(\partial_x {b})\big(\frac{h_1\ v}{h_1+\gamma h_2}\big)\big)^2\ -\ \frac\gamma2\big(h_1\partial_x\big(\frac{-h_2\ v}{h_1+\gamma h_2}\big) \big)^2 \\
 - \big(\frac{h_1\ v}{h_1+\gamma h_2}\big) \T[h_2,\beta b]\big(\frac{h_1\ v}{h_1+\gamma h_2}\big)+\gamma\big(\frac{-h_2\ v}{h_1+\gamma h_2}\big) \T[h_1,0]\big(\frac{-h_2\ v}{h_1+\gamma h_2}\big).
\end{multline*}

If, additionally, one assumes the bottom is flat (by setting $\beta=0$), then one recovers the following system, as already introduced in~\cite{Duchene13}:
\begin{equation}\label{eqn:GreenNaghdiMean}
\left\{ \begin{array}{l}
\displaystyle\partial_{ t}{\zeta} \ + \ \partial_x \Big(\frac{h_1h_2}{h_1+\gamma h_2}v\Big)\ =\ 0,  \\ \\
\displaystyle\partial_{ t}\Big( v \ + \ \mu\overline{\Q}[h_1,h_2]v \Big) \ + \ (\gamma+\delta)\partial_x{\zeta} \ + \ \frac{\epsilon}{2} \partial_x\Big(\dfrac{h_1^2 -\gamma h_2^2 }{(h_1+\gamma h_2)^2} |v|^2\Big) \ = \ \mu\epsilon\partial_x\big(\overline{\R}[h_1,h_2]v  \big) \\ \hfill -\frac{\gamma+\delta}{\Bo}\frac{\partial_x \big(k(\epsilon\sqrt\mu\zeta)\big)}{{\epsilon\sqrt\mu}}  \ ,
\end{array}
\right.
\end{equation}
where we define:

    \begin{align*}
        \overline{\Q}[h_1,h_2]V \  &\equiv \ \frac{-1}{3h_1 h_2}\Bigg(h_1 \partial_x \Big(h_2^3\partial_x\big(\frac{h_1\ V}{h_1+\gamma h_2} \big)\Big)\ +\ \gamma h_2\partial_x \Big( h_1^3\partial_x \big(\frac{h_2\ V}{h_1+\gamma h_2}\big)\Big)\Bigg), \\
        \overline{\R}[h_1,h_2]V  \  &\equiv \ \frac12  \Bigg( \Big( h_2\partial_x \big( \frac{h_1\ V}{h_1+\gamma h_2} \big)\Big)^2\ -\ \gamma\Big(h_1\partial_x \big(\frac{h_2\ V}{h_1+\gamma h_2} \big)\Big)^2\Bigg)\\
    &\qquad + \frac13\frac{V}{h_1+\gamma h_2}\ \Bigg( \frac{h_1}{h_2}\partial_x\Big( h_2^3\partial_x\big(\frac{h_1\ V}{h_1+\gamma h_2} \big)\Big) \ - \ \gamma\frac{h_2}{h_1}\partial_x\Big(h_1^3\partial_x \big(\frac{h_2\ V}{h_1+\gamma h_2}\big)\Big) \Bigg).
 \end{align*}
 
%\end{Remark}

Proposition~\ref{cons:GNv} thus generalizes the consistency result obtained in~\cite{Duchene13,DucheneIsrawiTalhouk} to the case $d=2$, and to non-flat topography.

\section{Lower order models}\label{sec:lowerorder}
The system of equations~\eqref{dtzetavsv},\eqref{eqn:GNGNv} is very broad, in the sense that it has been obtained with minimal assumptions: allowing $d=1$ and $d=2$, non-flat topography, and in the shallow water regime of Definition~\ref{Def:Regimes}. It is justified by a consistency result (Proposition~\ref{cons:GNv}). As argued in the introduction, the consistency result alone is not sufficient to fully justify a model. In particular, its well-posedness should be confirmed. As a matter of fact, contrarily to the water-wave case~\cite{Alvarez-SamaniegoLannes08}, the well-posedness of the Green-Naghi model in the bi-fluidic case is not clear, and similar systems have been proved to be ill-posed; see~\cite{LiskaMargolinWendroff95} and discussion in~\cite{CotterHolmPercival10}.  In the following subsections, we show that existing models in the literature directly descend from our Green-Naghdi model~\eqref{dtzetavsv},\eqref{eqn:GNGNv}, after additional assumptions (typically, restricting to $d=1$, flat bottom, and/or more stringent regimes), or with a lower precision. Their justification in the sense of consistency is therefore a direct application of Proposition~\ref{cons:GNv}, and stronger results (well-posedness, convergence) are stated when available.

\subsection{The Shallow water (Saint Venant) model}\label{sec:shallowwater}
In this section, we consider only the first order terms in equation~\eqref{eqn:GNGNv} (equivalently, we set $\mu=0$; this corresponds to the assumption that the horizontal velocity is constant across the vertical layers). The system~\eqref{dtzetavsv},\eqref{eqn:GNGNv}, withdrawing $\O(\mu)$ terms, is now simply
\begin{equation}\label{eqn:SW0}
\left\{ \begin{array}{l}
\partial_{ t}{\zeta} \ + \ \nabla\cdot \big(\dfrac{h_1h_2\ v}{h_1+\gamma h_2}\big) \ = \ 0,\\
\partial_{ t}(\t u_2-\gamma \t u_1) \ + \ (\gamma+\delta)\nabla{\zeta} \ + \ \frac{\epsilon}{2} \nabla\Big(|\t u_2 |^2 -\gamma |\t u_1|^2 \Big) 
= \  \frac{\gamma+\delta}{\Bo}\nabla \Delta\zeta ,
\end{array} \right.
\end{equation}
where we recall that $h_1\equiv 1-\epsilon\zeta,\ h_2\equiv \delta^{-1}+\epsilon\zeta$, $\t u_1\equiv -\mfQ[-\epsilon\zeta]\big( \frac{h_1h_2}{h_1+\gamma h_2}v\big)$ and $\t u_2\equiv \delta \mfQ[\epsilon\delta\zeta]\big( \frac{h_1h_2}{h_1+\gamma h_2}v\big)$, where the operator $\mfQ$ is defined in Definition~\ref{LemQ}.

A similar system is obtained
when using a different velocity variable, such as the shear velocity at the interface $V\equiv \nabla H^{\mu,\delta}\psi-\gamma \nabla \psi $. In that case, one obtains the following~\cite{BonaLannesSaut08}:
\begin{equation}\label{eqn:SW2}
\left\{ \begin{array}{l}
\partial_{ t}{\zeta} \ + \ \nabla\cdot (h_1 \mfR[\epsilon\zeta]  V ) \ = \ 0, \\
\partial_{ t} V \ + \ (\gamma+\delta)\nabla{\zeta} \ + \ \frac{\epsilon}{2} \nabla\Big(|V-\gamma \mfR[\epsilon\zeta]  V   |^2 -\gamma |\mfR[\epsilon\zeta]  V |^2 \Big) 
= \  \frac{\gamma+\delta}{\Bo}\nabla \Delta\zeta ,
\end{array} \right.
\end{equation}
where  $\mfR$ is defined similarly as $\mfQ$:  $\mfR[\epsilon\zeta] W$  is the only gradient solution to
\[ \nabla\cdot \big( (h_1+\gamma h_2) \mfR[\epsilon\zeta] W \big) \ = \ \nabla\cdot (h_2W) \ .\]
System (\ref{eqn:SW0}) and system (\ref{eqn:SW2}) are equivalent up to order $\O(\mu)$. In fact from proposition~\ref{prop:devu1u2v}, one has $V=\t u_2  - \gamma \t u_1+\O(\mu),$ and 
\begin{align*}
-\nabla\cdot \big( (h_1 +\gamma h_2)\t u_1\big) &=  -\nabla\cdot (h_1\t u_1)-\gamma \nabla\cdot (h_2\t u_1) \\
& = \nabla\cdot \big(h_2(\t u_2-\gamma \t u_1)\big)\\
& =  \nabla\cdot (h_2 V) +\O(\mu).
\end{align*}
Using the fact that $\t u_1$   is a gradient vector and the definition of operator $\mfR[\epsilon\zeta] $, we deduce that $-\t u_1= \mfR[\epsilon\zeta] V +\O(\mu)$ and thus $-\nabla\cdot (h_1\t u_1)=\nabla\cdot (h_1\mfR[\epsilon\zeta] V) +\O(\mu)$. The definition of $\t u_1$ implies now the equivalence between $(\ref{eqn:SW0})_1$ and $(\ref{eqn:SW2})_1$. In the same way to obtain the equivalence  of  $(\ref{eqn:SW0})_2$ and 
$(\ref{eqn:SW2})_2$ up to order $\O(\mu)$: we just use the two fact $\t u_1= -\mfR[\epsilon\zeta] V +\O(\mu)$ and $\t u_2= V-\gamma\mfR[\epsilon\zeta] V +\O(\mu)$.
\medskip

In one dimension ({\em i.e.} $d=1$), both~\eqref{eqn:SW0} and~\eqref{eqn:SW2} read
\begin{equation}\label{eqn:SW1}
\left\{ \begin{array}{l}
\partial_{ t}{\zeta} \ + \ \partial_x \big(\dfrac{h_1h_2\ v}{h_1+\gamma h_2} \big) \ = \ 0,\\
\partial_{ t} v \ + \ (\gamma+\delta)\partial_x {\zeta} \ + \ \frac{\epsilon}{2} \partial_x \Big(\frac{h_1^2-\gamma h_2^2}{(h_1+\gamma h_2)^2}|v |^2  \Big) 
= \  \frac{\gamma+\delta}{\Bo}\partial_x^3\zeta .
\end{array} \right.
\end{equation}

System~\eqref{eqn:SW2} has been derived and justified in the sense of consistency in~\cite{BonaLannesSaut08}, in the case of a flat bottom, and without surface tension. An equivalent consistency result clearly holds for~\eqref{eqn:SW0} (as a consequence of Proposition~\ref{cons:GNv} in particular). More precisely, one has
 \begin{Proposition}[Consistency]\label{prop:ConsSW}
Let $s\geq0$ and $U^\p\equiv(\zeta^\p,\psi^\p)$ be a family of solutions of the full Euler system~\eqref{eqn:EulerCompletAdim} 
for which $(\zeta^\p,\nabla\psi^\p)^\top$ is bounded in $L^\infty([0,T);H^{s+N})^{d+1}$ with sufficiently large $N$; 
and such that there exists $h_0>0$ with
\[
h_1 \ \equiv\  1-\epsilon\zeta^\p \geq h_0>0, \quad h_2 \ \equiv \ \frac1\delta +\epsilon \zeta^\p\geq h_0>0.
\]
 Define $v$ as in~ definition~\ref{defV}, and $V\equiv \nabla H^{\mu,\delta}\psi -\gamma \nabla \psi $. 

Then $(\zeta,v)^\top$ satisfies~\eqref{eqn:SW0}, up to a remainder $(0,R_1)^\top$; and $(\zeta,V)^\top$ satisfies~\eqref{eqn:SW2}, up to a remainder $(0,R_2)^\top$; with
\[ \big\Vert R_1 \big\Vert_{L^\infty([0,T);H^s)^d} \ + \ \big\Vert R_2 \big\Vert_{L^\infty([0,T);H^s)^d} \ \leq \ C\ \mu,\]
with $C=C(h_0^{-1},\mu_{\max},\delta_{\min}^{-1},\delta_{\max},\big\Vert \zeta^\p \big\Vert_{L^\infty([0,T);H^{s+N})},\big\Vert \nabla\psi^\p \big\Vert_{L^\infty([0,T);H^{s+N})^d})$,  uniform with respect to the parameters $\p\in\P_{SW}$.
\end{Proposition}

Let us now mention that the well-posedness of the Cauchy problem for~\eqref{eqn:SW2} (without surface tension and bottom topography) has been studied in~\cite{GuyenneLannesSaut10}, and we reproduce below their result (see~\cite[Theorem~2]{GuyenneLannesSaut10}).
\begin{Proposition}[Well-posedness] Let $d=2$ and $\Bo^{-1}=\beta=0$.
Let $s\geq t_0+1$, $t_0 > 1$, and $U^0 = (\zeta^0,V^0)^\top \in H^s(\RR^2)^{3}$ be such that there exists $h_0>0$ with
\begin{equation}\label{conditionSW}
\min\Big\{ \ 1-\epsilon\big\vert\zeta^0\big\vert_{L^\infty} \ , \ \frac1\delta -\epsilon\big\vert\zeta^0\big\vert_{L^\infty} \ , \  \gamma+\delta-\gamma \frac{\epsilon^2\vert V+(1-\gamma)\mfR[\epsilon\zeta]V\vert^2}{1-\epsilon\big\vert\zeta^0\big\vert_{L^\infty}+\gamma(\delta^{-1}+\epsilon\big\vert\zeta^0\big\vert_{L^\infty})} \ \Big\}\ \geq h_0\   >\ 0,
\end{equation}
 and $\curl V^0 = 0$. Then, there exists $T_{\max} > 0$ and a unique maximal solution $U = (\zeta,V)^\top \in C([0,T_{\max}/\epsilon);H^s(\RR^2)^{3})$  to~\eqref{eqn:SW2} with initial condition $U^0$. 
\end{Proposition}
\begin{Remark} In the case $d=1$, the conditions~\eqref{conditionSW} can be replaced by
\[ h_1 \ \equiv\  1-\epsilon\zeta \geq h_0>0, \quad h_2 \ \equiv \ \frac1\delta +\epsilon \zeta\geq h_0>0, \quad \gamma+\delta-\gamma \frac{\epsilon^2(h_1+h_2)^2}{(h_1+\gamma h_2)^3}|v|^2\geq h_0>0,\]
which is sufficient to write~\eqref{eqn:SW1} as a symmetrizable quasilinear equation; thus classical techniques apply. See~\cite[Theorem~1]{GuyenneLannesSaut10} for the precise result.
\end{Remark}

\subsection{A Green-Naghdi type model in the Camassa-Holm regime}\label{sec:Serre}
The present section is limited to the so-called Camassa-Holm regime (that is
using additional assumptions $\epsilon=\O(\sqrt\mu)$; see Definition~\ref{Def:Regimes}), flat bottom  case ($\beta=0$) and one dimensional space $d=1$. Moreover, we neglect the surface tension component in our model.
All the estimates are now understood uniformly with respect to $(\mu,\epsilon,\delta,\gamma)\in\P_{CH}$.
\medskip

In that case, one can easily check that the following approximations hold for $\overline{\Q}$, $\overline{\R}$ defined in section~\ref{sec:GN}
 \begin{align*}
        \overline{\Q}[h_1,h_2]V \ &= \ - \nu\partial_x^2 V-\epsilon\frac{\gamma+\delta}3 \left((\beta-\alpha)V \partial_x^2\zeta \ + \ (\alpha+2\beta)\partial_x(\zeta\partial_x V)-\beta\zeta\partial_x^2 V\right) \ + \ \O(\mu), \\
        \overline{\R}[h_1,h_2]V  \  &= \ \alpha\left(\frac12 (\partial_x V)^2+\frac13 V \partial_x^2 V\right) \ + \ \O(\sqrt\mu)
 \end{align*}
with \begin{equation}\label{eqn:defnualphabeta}
\nu=\frac{1+\gamma\delta}{3\delta(\gamma+\delta)}\ , \quad \alpha=\dfrac{1-\gamma}{(\gamma+\delta)^2} \quad \text{ and } \quad \beta=\dfrac{(1+\gamma\delta)(\delta^2-\gamma)}{\delta(\gamma+\delta)^3}\ .
\end{equation}

Plugging these expansions into system~\eqref{eqn:GreenNaghdiMean} yields a simplified model, precise with the same order of magnitude as the original model (that is $\O(\mu^2)$) in the Camassa-Holm regime. Furthermore, after several additional transformations, ones may produce an equivalent model (again, in the sense of consistency) which possesses a structure similar to symmetrizable quasilinear systems, thus allowing its full justification. The following system of equations has been introduced and justified by the authors in~\cite{DucheneIsrawiTalhouk}.  
\begin{equation}\label{eq:Serre2mr}\left\{ \begin{array}{l}
\partial_{ t}\zeta +\partial_x\left(\dfrac{h_1 h_2}{h_1+\gamma h_2} v\right)\ =\  0,\\ \\
\mfT[\epsilon\zeta] \left( \partial_{ t}  v + \epsilon \varsigma { v } \partial_x { v } \right) + (\gamma+\delta)q_1(\epsilon\zeta)\partial_x \zeta   \\
\qquad \qquad+\frac\epsilon2 q_1(\epsilon\zeta) \partial_x \left(\frac{h_1^2  -\gamma h_2^2 }{(h_1+\gamma h_2)^2}| v|^2-\varsigma | v|^2\right)=  -  \mu \epsilon\frac23\alpha \partial_x\big((\partial_x  v)^2\big) ,
\end{array} \right. \end{equation}

where 

\begin{equation}\label{defmfT}
{\mfT}[\epsilon\zeta]V \ = \ q_1(\epsilon\zeta)V \ - \ \mu\nu \partial_x \Big(q_2(\epsilon\zeta)\partial_xV \Big),
\end{equation}
with $q_i(X)\equiv 1+\kappa_i X $ ($i=1,2$) and $\kappa_1,\kappa_2,\varsigma$ are defined by :
\begin{equation}\label{defkappa}
\kappa_1 \ = \ \frac{\gamma+\delta}{3\nu}(2\beta-\alpha)  \ = \ 2\dfrac{\delta^2-\gamma}{\gamma+\delta}-\delta\dfrac{1-\gamma}{1+\gamma\delta}, \quad \kappa_2 \ = \ \frac{(\gamma+\delta)\beta}{\nu} \ = \ 3\dfrac{\delta^2-\gamma}{\gamma+\delta},\\
\end{equation}
\begin{equation}\label{defvarsigma}
\varsigma \ = \ \frac{2\alpha-\beta}{3\nu} \ = \ 2\delta\dfrac{1-\gamma}{(\gamma+\delta)(1+\gamma\delta)}-\dfrac{\delta^2-\gamma}{(\gamma+\delta)^2}.
\end{equation}
We recall that the shear mean velocity is uniquely defined in the case of one dimension, $d=1$ (see (\ref{eqn:defbarv}) and (\ref{idGvsU1U2})) by
\begin{equation}\label{eqn:defbarvmr}
\frac1\mu G^{\mu,\delta}[\epsilon\zeta]\psi \ = \ -\partial_x \Big(\frac{h_1h_2}{h_1+\gamma h_2} v\Big). 
\end{equation}
\medskip

System~\eqref{eq:Serre2mr} is fully justified as an asymptotic model by the following results (see~\cite{DucheneIsrawiTalhouk} for a more precise statement):
 \begin{Proposition}[Consistency]
Let $U^\p\equiv(\zeta^\p,\psi^\p)$ be a family of solutions of the full Euler system~\eqref{eqn:EulerCompletAdim} 
for which $(\zeta^\p,\partial_x\psi^\p)^\top$ is bounded in $L^\infty([0,T);H^{s+N})^{2}$ with $s\geq0$ and sufficiently large $N$, and uniformly with respect to $(\mu,\epsilon,\delta,\gamma)\equiv \p\in \P_{CH}$; see Definition~\ref{Def:Regimes}. Moreover, assume 
\begin{equation}\label{depthcond}\tag{H1}
\exists h_{01}>0 \text{ such that }\quad h_1 \ \equiv\  1-\epsilon\zeta^\p \geq\ h_{01}\ >\ 0, \quad h_2 \ \equiv \ \frac1\delta +\epsilon \zeta^\p\geq\ h_{01}\ >\ 0.
\end{equation}
 Define $v^\p$ as in~\eqref{eqn:defbarvmr}.
Then $(\zeta^\p, v^\p)^\top$ satisfies~\eqref{eq:Serre2mr}, up to a remainder $(0,R)^\top$, bounded by
\[ \big\Vert R \big\Vert_{L^\infty([0,T);H^s)} \ \leq \ (\mu^2+\Bo^{-1})\ C,\]
with $C=C(h_{01}^{-1},M,\mu_{\max},\frac1{\delta_{\min}},\delta_{\max},\big\Vert (\zeta^\p , \partial_x\psi^\p)^\top \big\Vert_{L^\infty([0,T);H^{s+N})^{2}}))$,  uniform with respect to the parameters $\p\in\P_{CH}$.
\end{Proposition}
\medskip

\noindent System~\eqref{eq:Serre2mr} is well-posed (in the sense of Hadamard) in the energy space $X^s\ = \ H^s(\RR)\times H^{s+1}(\RR)$, endowed with the norm
\[
\forall\; U=(\zeta,v)^\top \in X^s, \quad \vert U\vert^2_{X^s}\equiv \vert \zeta\vert^2 _{H^s}+\vert v\vert^2 _{H^s}+ \mu\vert \partial_xv\vert^2 _{H^s},
\]
provided that the following ellipticity condition (for the operator $\mfT$) holds:
\begin{equation}\label{CondEllipticity}\tag{H2}
\exists h_{02}>0 \text{ such that }\quad \inf_{x\in \RR} \left(1+\epsilon\kappa_2\zeta\right) \ge \ h_{02} \ > \ 0 \ ; \qquad    \inf_{x\in \RR}  \left( 1+\epsilon \kappa_1\zeta  \right)\ge \ h_{02} \ > \ 0.
\end{equation}
\begin{Theorem}[Existence and uniqueness]\label{thbi1mr}
	Let   $s_0>1/2$, $s\geq s_0+1$, and let
	 $U_0=(\zeta_0,v_0)^\top\in X^s$ satisfy~\eqref{depthcond},\eqref{CondEllipticity}. Then there exists
         a maximal time $T_{\max}>0$, uniformly bounded from below with respect to $\p\in \P_{CH} $, such that the system of
         equations~\eqref{eq:Serre2mr} admits
	 a unique solution $U=(\zeta,v)^\top \in C^0([0,T_{\max}/\epsilon);X^s)\cap C^1([0,T_{\max}/\epsilon);X^{s-1})$ with the initial value $(\zeta,v)\id{t=0}=(\zeta_0,v_0)$
         and preserving the conditions ~\eqref{depthcond},\eqref{CondEllipticity} (with different lower bounds) for any $t\in [0,T_{\max}/\epsilon)$.
         
   Moreover, for any $0\leq T<T_{\max}$, there exists $C_0,\lambda_T=  C(h_{01}^{-1},h_{02}^{-1},\delta_{\max},\delta_{\min}^{-1},M,T,\big\vert U_0\big\vert_{X^{s}})$, independent of $\p\in\P_{CH}$, such that one has the energy estimate
\[\forall\ 0\leq t\leq\frac{T}{\epsilon}\ , \qquad 
\big\vert U(t,\cdot)\big\vert_{X^{s}} \ + \ 
\big\vert \partial_t U(t,\cdot)\big\vert_{X^{s-1}}  \leq C_0 e^{\epsilon\lambda_{T} t}\ .
\]
If $T_{\max}<\infty$, one has
         \[ \vert U(t,\cdot)\vert_{X^{s}}\longrightarrow\infty\quad\hbox{as}\quad t\longrightarrow \frac{T_{\max}}{\epsilon},\]
         or one of the two conditions~\eqref{depthcond},\eqref{CondEllipticity} ceases to be true $\hbox{as}\quad t\longrightarrow T_{\max}/{\epsilon}$.
\end{Theorem}
\begin{Proposition}[Stability] \label{th:stabilityWPmr}
Let $s\geq s_0+1$ with  $s_0>1/2$, and let
	 $U_{0,1}=(\zeta_{0,1},v_{0,1})^\top\in X^{s}$ and $U_{0,2}=(\zeta_{0,2},v_{0,2})^\top\in X^{s+1}$. Under the assumptions of Theorem~\ref{thbi1mr}, let $U_j$ be the solution of system~\eqref{eq:Serre2mr} with $U_j\id{t=0}=U_{0,j}$.Then there exists $T,\lambda,C_0= C(h_{01}^{-1},h_{02}^{-1},\delta_{\max},\delta_{\min}^{-1},M,\big\vert U_{0,1}\big\vert_{X^s},\vert U_{0,2}\vert_{X^{s+1}})>0$ such that
	  \begin{equation*}
	 	\forall t\in [0,\frac{T}{\epsilon}],\qquad
	 	\big\vert (U_1-U_2)(t,\cdot)\big\vert_{X^s} \leq C_0 e^{\epsilon\lambda_{T} t} \big\vert U_{1,0}-U_{2,0}\big\vert_{X^s}.
	 \end{equation*}
\end{Proposition}

Finally, the following ``convergence result'' states that the solutions of our system approach the solutions of the full Euler system, with as good a precision as $\mu$ (and $\Bo^{-1}$) is small.

\begin{Theorem}[Convergence] \label{prop:stability}
Let $\p\in \P_{CH}$ and $U^0\equiv(\zeta^0,\psi^0)^\top\in H^{s+N}$, $N$ sufficiently large, satisfy the hypotheses of Theorem~5 in~\cite{Lannes13} (see Theorem~\ref{th:WPEuler}), as well as~\eqref{depthcond},\eqref{CondEllipticity}. Then there exists $C,T>0$, independent of $\p$, such that
\begin{itemize}
\item There exists a unique solution $U\equiv (\zeta,\psi)^\top$ to the full Euler system~\eqref{eqn:EulerCompletAdim}, defined on $[0,T]$ and with initial data $(\zeta^0,\psi^0)^\top$ (provided by Theorem~5 in~\cite{Lannes13});
\item There exists a unique solution $U_a\equiv (\zeta_a,v_a)^\top$ to our new model~\eqref{eq:Serre2mr}, defined on $[0,T]$ and with initial data $(\zeta^0,v^0)^\top$, with $v^0\equiv  v[\zeta^0,\psi^0]$ defined as in~\eqref{eqn:defbarvmr} (provided by Theorem~\ref{thbi1mr});
\item With $ v\equiv v[\zeta,\psi]$, defined as in~\eqref{eqn:defbarvmr},one has
\[ \big\vert (\zeta, v)-(\zeta_a,v_a) \big\vert_{L^\infty([0,T];X^s)}\leq C(\mu^2+\Bo^{-1}) t.\]
\end{itemize}
The above results hold on time interval $t\in[0,T/\epsilon]$ with $T$ bounded by below, independently of $\p\in\P_{CH}$, provided that a stronger criterion is satisfied by the initial data. This corresponds to setting $\varrho=1$ in Theorem~\ref{th:WPEuler}; see criterion (5.5) and Theorem~6 in~\cite{Lannes13} for the precise statement.
\end{Theorem}

\begin{Remark}
The new model allows to fully justify any well-posed system, consistent with our model~\eqref{eq:Serre2mr}, thanks to an a priori estimate between two approximate solutions of our system, established in~\cite{DucheneIsrawiTalhouk} (see Proposition 7.1 and Theorem 7.5 therein). This is used in particular to obtain the convergence results of the unidirectional and decoupled approximations stated in Section~\ref{sec:scalar}.
\end{Remark}

\begin{Remark}[The model with surface tension]
The previous results concerning~\eqref{eq:Serre2mr}, where the surface tension effects are neglected, still hold when surface tension is taken into account; see~\eqref{eqn:EulerCompletAdim}. These results will be precisely stated in a forthcoming paper.
\end{Remark}

\subsection{Boussinesq models}\label{sec:Boussinesq}
In this section, we restrict ourselves to the case of flat bottom ($\beta=0$), and unidimensional case ($d=1$). Moreover, we restrict the regime under study to the so-called long wave regime, where $\epsilon=\O(\mu)$; see Definition~\ref{Def:Regimes}. We also assume that the surface tension term is at most of size $\Bo^{-1}=\O(\mu)$, with the following:
\[ \Bo^{-1} \ = \ \mu \bo^{-1} \quad \text{ with } \quad \bo\equiv \dsp\frac{g(\rho_2-\rho_1)d_1^2}{\sigma} \in [\bo_{\min},\infty).\]
 In that case, when withdrawing $\O(\mu^2)$ terms, one can easily check (see the proof of Proposition~\ref{cons:Bouss}, below for detailed calculations) that the system~\eqref{eqn:GreenNaghdiMean} becomes a simple quasilinear system, with additional linear dispersive terms:
\begin{equation}\label{eqn:BoussinesqMean}
\displaystyle\partial_{ t}U -\mu A_1\partial_x^2\partial_t U \ + \ A_0\partial_x U \ + \ \epsilon A[U]\partial_xU \ + \ \mu A_2\partial_x^3 U \ =\ 0, 
\end{equation}
with
\[ A_0 = \begin{pmatrix}
0&\frac{1}{\gamma+\delta}\\ \gamma+\delta&0 \end{pmatrix}, \quad 
A\begin{bmatrix}\zeta\\ v \end{bmatrix} = \frac{\delta^2-\gamma}{(\gamma+\delta)^2}\begin{pmatrix}
v&\zeta \\ 0&v \end{pmatrix}, \quad
A_1 = \begin{pmatrix}
0&0\\ 0&\frac{1+\gamma\delta}{3\delta(\gamma+\delta)} \end{pmatrix}, \quad
 A_2 =  \begin{pmatrix}
0&0\\ -\frac{\gamma+\delta}{\bo}&0
\end{pmatrix}. \]
The full Euler system is consistent with this model, with the same precision as the Green-Naghdi models, provided that the assumptions of the long wave regime, and in particular $\epsilon\leq M\mu$, hold (in addition to the aforementioned $\beta=0,d=1$); see Proposition~\ref{cons:Bouss}, below.

\paragraph{Boussinesq systems with improved frequency dispersion.} Let us emphasize that system~\eqref{eqn:BoussinesqMean} is only one of a large family of Boussinesq-type models, that are consistent with precision $\O(\mu^2)$.
Briefly, one can make use of
\begin{itemize}
\item {\em Near-identity change of variable.} Define, for $\theta_1,\theta_2\geq 0$, the following
\[ v_{\theta_1,\theta_2} \ = \  (I - \mu \theta_1 \partial_x^2)^{-1}(I - \mu\theta_2 \partial_x^2)  v.\]
When rewriting~\eqref{eqn:BoussinesqMean} with respect to this new variable, $v_{\theta_1,\theta_2}$, and withdrawing $\O(\mu^2)$ terms, one obtains
\begin{equation}\label{eqn:Boussinesqparam}
\partial_{ t}U_{\theta_1,\theta_2} -\mu \t A_1\partial_x^2\partial_t U_{\theta_1,\theta_2} \ + \ A_0\partial_x U_{\theta_1,\theta_2} \ + \ \epsilon A[U_{\theta_1,\theta_2}]\partial_x U_{\theta_1,\theta_2} \ + \ \mu\t A_2\partial_x^3 U_{\theta_1,\theta_2}  \ =\ 0, 
\end{equation}
with $U_{\theta_1,\theta_2}\equiv (\zeta,v_{\theta_1,\theta_2})^\top$ and 
\[\t A_1 \ = \ \begin{pmatrix}
\theta_2&0\\ 0&\frac{1+\gamma\delta}{3\delta(\gamma+\delta)}+\theta_1 \end{pmatrix}, \quad
\t A_2  =   \begin{pmatrix}
0&\frac{-\theta_1}{\gamma+\delta}\\ -\frac{\gamma+\delta}{\bo}-\theta_2(\gamma+\delta)&0
\end{pmatrix}. \]
\item {\em The Benjamin-Bona-Mahony trick.} Using that $\partial_t U_{\theta_1,\theta_2} \ + \ A_0\partial_x U_{\theta_1,\theta_2} \ = \ \O(\mu)$, one has for any $\lambda_1,\lambda_2\in\RR$:
\[ \partial_t U_{\theta_1,\theta_2} \ = \  \begin{pmatrix}
1-\lambda_1&0\\ 0&1-\lambda_2 \end{pmatrix} \partial_t U_{\theta_1,\theta_2} - \begin{pmatrix}
\lambda_1&0\\ 0&\lambda_2 \end{pmatrix} A_0\partial_x U_{\theta_1,\theta_2} +\O(\mu^2). 
\]
Plugging this approximation into~\eqref{eqn:Boussinesqparam} yields, withdrawing again $\O(\mu^2)$ terms,
\begin{equation}\label{eqn:Boussinesqparam2}
\partial_{ t}U_{\theta_1,\theta_2} -\mu \breve A_1\partial_x^2\partial_t U_{\theta_1,\theta_2} \ + \ A_0\partial_x U_{\theta_1,\theta_2} \ + \ \epsilon A[U_{\theta_1,\theta_2}]\partial_x U_{\theta_1,\theta_2} \ + \ \mu\breve A_2\partial_x^3 U_{\theta_1,\theta_2}  \ =\ 0, 
\end{equation}
with
\[\breve A_1 \ = \ \begin{pmatrix}
(1-\lambda_1)\theta_2&0\\ 0&(1-\lambda_2)\big(\frac{1+\gamma\delta}{3\delta(\gamma+\delta)}+\theta_1\big) \end{pmatrix}, \]
\[
\breve A_2  =   \begin{pmatrix}
0&\frac{\lambda_1\theta_2-\theta_1}{\gamma+\delta}\\ -\frac{\gamma+\delta}{\bo}+(\gamma+\delta)\left[\lambda_2\big(\frac{1+\gamma\delta}{3\delta(\gamma+\delta)}+\theta_1\big)-\theta_2\right]&0
\end{pmatrix}. \]
\end{itemize}

\noindent The above transformations are useful, for example, with the aim of {\em improving the frequency dispersion}, that is choosing the coefficients so that the dispersion relation fits the one of the full Euler system with high precision, even for relatively large $\mu$ (using truncated Taylor series or Pad\'e approximant). It may also be useful for mathematical purposes to generate such a large family of models, which are all equivalent in the sense of consistency, but may have very different properties (well-posedness, integrability, {\em etc.}). We let the reader refer to~\cite{Lannes} for a more detailed account.

In~\cite{SautXu12}, Saut and Xu offer an in-depth study of the well-posedness of such Boussinesq systems in the water-wave setting (thus with different coefficients). It would be interesting, but out of the scope of the present work, to adapt the techniques developed therein to our bi-fluidic systems~\eqref{eqn:Boussinesqparam2}.

\paragraph{A fully justified symmetric Boussinesq model.} Another strategy for constructing a model with improved properties, that has been used in~\cite{BonaColinLannes05,Duchene11a} and that we develop in the following, consists in symmetrizing the original model~\eqref{eqn:BoussinesqMean}, up to precision $\O(\mu^2)$.
Define
\[ S_0 \ \equiv \ \begin{pmatrix}
\gamma+\delta&0\\ 0&\frac{1}{\gamma+\delta}\end{pmatrix},\quad  S\begin{bmatrix}\zeta\\ v \end{bmatrix}\equiv  \frac{\delta^2-\gamma}{(\gamma+\delta)^2}\begin{pmatrix}
0&0 \\ 0&\zeta
\end{pmatrix}, \quad T \ \equiv \ \begin{pmatrix}
(\gamma+\delta)(1+\frac1\bo)&0\\ 0&\frac1{\gamma+\delta}
\end{pmatrix} \ .
\]
Multiplying~\eqref{eqn:BoussinesqMean} with $(S_0+\epsilon S[U]-\mu T\partial_x^2)$ and withdrawing $\O(\mu^2)$ terms yields
\begin{equation}\label{eqn:BoussinesqSym}
(S_0+\epsilon S[U]-\mu S_1\partial_x^2 ) \partial_{ t}U  \ + \ (\Sigma_0+\epsilon \Sigma[U]-\mu \Sigma_1\partial_x^2)\partial_x U  \ =\ 0, 
\end{equation}
with the following {\em symmetric} matrices and operators:
\[ \Sigma_0\equiv S_0A_0=\begin{pmatrix}
0&1\\1&0
\end{pmatrix},\quad \Sigma\begin{bmatrix}\zeta\\ v \end{bmatrix}\equiv S_0 A\begin{bmatrix}\zeta\\ v \end{bmatrix}+S\begin{bmatrix}\zeta\\ v \end{bmatrix}A_0 \ = \  \frac{\delta^2-\gamma}{\gamma+\delta}\begin{pmatrix}
v&\zeta \\ \zeta&\frac{v}{(\gamma+\delta)^2}
\end{pmatrix}, \]
\[ S_1  \equiv S_0 A_1+T =  \begin{pmatrix}
(\gamma+\delta)(1+\frac1\bo) &0\\ 0&\frac{1}{\gamma+\delta}+ \frac{1+\gamma\delta}{3\delta(\gamma+\delta)^2} 
\end{pmatrix}, \quad \Sigma_1  \equiv  T A_0 -S_0 A_2  =   \begin{pmatrix}
0&1+\frac1\bo\\ 1+\frac1\bo&0
\end{pmatrix}.
\]

This new model is fully justified, in the sense described in the introduction. Let us detail the consistency, well-posedness and convergence results below.

\begin{Proposition}[Consistency] \label{cons:Bouss}
Let $U^\p\equiv(\zeta^\p,\psi^\p)^\top$ be a family of solutions to the full Euler system~\eqref{eqn:EulerCompletAdim} with $\beta=0$ and $d=1$, such that there exists $T>0$, $s\geq0$ for which $(\zeta^\p,\partial_x \psi^\p)^\top$ is bounded in $L^\infty([0,T);H^{s+N})^{2}$ ($N$ sufficiently large), uniformly with respect to $(\mu,\epsilon,\delta,\gamma)\equiv \p\in \P_{LW}$; see Definition~\ref{Def:Regimes}, and $\bo^{-1}\leq \bo_{\min}^{-1}$, $\bo_{\min}>0$. Moreover, assume that
\[
\exists h_{0}>0 \text{ such that }\quad h_1\equiv  1-\epsilon \zeta^\p \geq\ h_{0}\ >\ 0, \quad  h_2\equiv \frac1\delta + \epsilon  \zeta^\p \geq\ h_{0}\ >\ 0.
\]
Define $v^\p\equiv \b u_2-\gamma \b u_1$; see Definition~\ref{defU1U2}. Then $(\zeta^\p,  v^\p)^\top$  satisfy~\eqref{eqn:BoussinesqMean} (resp.~\eqref{eqn:BoussinesqSym}) up to a remainder term, $R_B$ (resp. $R_S$), bounded by
\[ \big\Vert R_B \big\Vert_{L^\infty([0,T);H^s)^2}  \ + \ \big\Vert R_S \big\Vert_{L^\infty([0,T);H^s)^2} \ \leq \ \mu^2\ C \ ,\]
with $C=C(h_0^{-1},\bo_{\min}^{-1},\mu_{\max},M,\delta_{\min}^{-1},\delta_{\max},\big\Vert \zeta^\p\big\Vert_{L^\infty([0,T);H^{s+N})} ,\big\Vert\partial_x\psi^\p\big\Vert_{L^\infty([0,T);H^{s+N})})$, uniform with respect to $(\mu,\epsilon,\delta,\gamma)\in \P_{LW}$.
\end{Proposition}
\begin{proof} Let us first recall that since we assume $d=1$, then defining $ v^\p\equiv \b u_2-\gamma \b u_1$ is equivalent as defining $ v^\p$ through Definition~\ref{defV}, but requires only the non-vanishing depth condition $h_1,h_2\geq h_0>0$, instead of the more stringent condition of Definition~\ref{defV}. Let us also emphasize that system~\eqref{dtzetavsv},\eqref{eqn:GNGNv} is {\em exactly} (that is, without any approximation) system~\eqref{eqn:GreenNaghdiMean}, in the case $\beta=0$ and $d=1$. Thus,
by Proposition~\ref{cons:GNv}, we know that $(\zeta^\p, v^\p)^\top$ satisfies~\eqref{eqn:GreenNaghdiMean} up to a remainder term, $R$, satisfying
\[ \big\Vert R \big\Vert_{L^\infty([0,T);H^s)^2}  \ \leq \ \mu^2\ C, \]
with $C$ as in the Proposition~\ref{cons:Bouss}.

Thus one only needs to check that the neglected terms from~\eqref{eqn:GreenNaghdiMean} in~\eqref{eqn:BoussinesqMean} (resp.~\eqref{eqn:BoussinesqSym}), when using that $\epsilon \leq M\mu$, are estimated in the same way. Let us detail briefly.

We first claim that the following holds, for any $s\geq 0$ and with $t_0>1/2$:
\begin{equation}\label{est1}
\big\vert \frac{h_1h_2}{h_1+\gamma h_2} \ - \ \frac{1}{\gamma+\delta} \ - \ \frac{\delta^2-\gamma}{(\gamma+\delta)^2}\epsilon\zeta \big\vert_{H^s} \ \leq \ \epsilon^2 C (h_0^{-1},\delta_{\max},\delta_{\min}^{-1},\big\vert \zeta \big\vert_{H^{\max{\{t_0,s\}}}} ).
\end{equation}
The formal expansion (as powers of $\epsilon\zeta$) is easily checked, so that the estimate in $L^\infty$ norm is straightforward: 
\begin{align*}
\big\vert \frac{h_1h_2}{h_1+\gamma h_2}  -  \frac{1}{\gamma+\delta}  -  \frac{\delta^2-\gamma}{(\gamma+\delta)^2}\epsilon\zeta \big\vert_{L^\infty} \ &= \  \big\vert  \frac{(\delta^2-\gamma)(1-\gamma)}{(\gamma+\delta)^2}\frac{(\epsilon\zeta)^2}{h_1+\gamma h_2}  \big\vert_{L^\infty}  \\
 &\leq \ \epsilon^2\  C (h_0^{-1},\delta_{\max},\delta_{\min}^{-1},\big\vert \zeta \big\vert_{L^{\infty}} ).
\end{align*}
The control in $H^s$ is slightly more elaborate, as $h_1,h_2$ are not bounded in $H^s$, since they do not decrease at infinity. We let the reader refer to~\cite[Lemma~4.5]{DucheneIsrawiTalhouk} to see how this technical difficulty may be faced. 

It follows from~\eqref{est1} that
\[
 \partial_x\big( \frac{h_1h_2}{h_1+\gamma h_2} v \big) \ = \ \frac{1}{\gamma+\delta} \partial_x  v \ + \ \epsilon\  \frac{\delta^2-\gamma}{(\gamma+\delta)^2}\partial_x \big(\zeta\  v\big) \ + \ \epsilon^2 R_1 ,
\]
 with $\big\vert R_1 \big\vert_{H^s} \ \leq \  C (h_0^{-1},\delta_{\max},\delta_{\min}^{-1},\big\vert \zeta \big\vert_{H^{\max{\{t_0,s\}}+1}},\big\vert  v \big\vert_{H^{\max{\{t_0,s\}}+1}} ).$

One obtains in the same way the following estimates:
\[
 \partial_x \Big(\dfrac{h_1^2 -\gamma h_2^2 }{(h_1+\gamma h_2)^2}  v^2\Big) \ = \ \partial_x \Big(\frac{\delta^2-\gamma}{(\gamma+\delta)^2}  v^2\Big) + \epsilon R_2,  \]
with $\big\vert R_2 \big\vert_{H^s} \ \leq \  C (h_0^{-1},\delta_{\max},\delta_{\min}^{-1},\big\vert \zeta \big\vert_{H^{\max{\{t_0,s\}}+1}},\big\vert  v \big\vert_{H^{\max{\{t_0,s\}}+1}} )$, and
 \[
 \partial_t\big( \overline{\Q}[h_1,h_2] v\big)  \ = \ - \frac{1+\gamma\delta}{3\delta(\gamma+\delta)} \partial_x^2\partial_t v \ + \ \epsilon R_3 \ , \quad \quad  \partial_x \Big(
\big\vert\overline{\R}[h_1,h_2] v \Big) \ \equiv\  R_4,
\]
with $\big\vert R_3 \big\vert_{H^s} =C (h_0^{-1},\delta_{\max},\delta_{\min}^{-1},\big\vert \zeta \big\vert_{H^{\max{\{t_0,s\}}+3}} ,\big\vert \partial_t \zeta \big\vert_{H^{\max{\{t_0,s\}}+2}} ,\big\vert  v \big\vert_{H^{\max{\{t_0,s\}}+3}} ,\big\vert \partial_t v \big\vert_{H^{\max{\{t_0,s\}}+2}} ) $ and $\big\vert R_4 \big\vert_{H^s} =C (h_0^{-1},\delta_{\max},\delta_{\min}^{-1},\big\vert \zeta \big\vert_{H^{\max{\{t_0,s\}}+3}}  ,\big\vert  v \big\vert_{H^{\max{\{t_0,s\}}+3}}) $.

Let us recall that estimates on $\partial_t\zeta$ and $\partial_t v$ may be deduced from estimates on $\zeta, v$ as they satisfy the full Euler system~\eqref{eqn:EulerCompletAdim}, allowing for a loss of derivatives.

Altogether, choosing $N$ sufficiently large and since $\epsilon\leq M\mu$ when $\p\in \P_{LW}$, it follows that $(\zeta, v)^\top$ satisfy~\eqref{eqn:BoussinesqMean} up to a remainder term, denoted $R_B$, and bounded by
\[ \big\Vert R_B \big\Vert_{L^\infty([0,T);H^s)^2} \ \leq \ \big\Vert R \big\Vert_{L^\infty([0,T);H^s)^2}+\epsilon^2\big\Vert |R_1|+|R_2|+|R_3|+| R_4| \big\Vert_{L^\infty([0,T);H^s)^2} \ \leq  \mu^2\ C .\]

 One obtains similarly the estimate concerning the symmetric system~\eqref{eqn:BoussinesqMean}, after controlling the extra error terms:
\begin{multline*} \big\Vert R_S \big\Vert_{L^\infty([0,T);H^s)^2}  \ \leq \ \big\Vert R_B \big\Vert_{L^\infty([0,T);H^s)^2}  \\ + \ \big\Vert \big(\epsilon S[U]-\mu T\partial_x^2\big)\big( -\mu A_1\partial_x^2\partial_t U  + \epsilon A[U]\partial_xU  +  \mu A_2\partial_x^3 U\big)  \big\Vert_{L^\infty([0,T);H^s)^2} .\end{multline*}
This is easily checked if $N$ is sufficiently large, and using again that $\epsilon \leq M\mu$. Proposition~\ref{cons:Bouss} is proved.
\end{proof}

In addition to its justification in the sense of consistency, one is able to fully justify the symmetric Boussinesq model~\eqref{eqn:BoussinesqSym}, thanks to its following properties:
\begin{enumerate}
\item The matrices $S_0$, $\Sigma_0$, $S_1$, $\Sigma_1$ are $2$-by-$2$ {\em symmetric} matrices;
\item $S[\cdot]$ and $\Sigma[\cdot]$ are linear mappings with values into $2$-by-$2$ {\em symmetric} matrices;
\item $S_0$ and $S_2$ are {\em definite positive}.
\end{enumerate}
These properties allow to control a natural energy of the system:
\[ E^s(U) \ \equiv \ \big(\ S_0 U \ , \ U \ \big) \ + \ \epsilon \big(\ S[\underline{U}] U \ , \  U\ \big) \ + \ \mu \big(\ S_1\partial_x U \ , \ \partial_x U\ \big),\]
which is equivalent to the scaled Sobolev norm $X^{s+1}_{\mu}$:
\[ \big\vert \ U \ \big\vert_{X^{s+1}_{\mu}}^2 \ = \ \big\vert \ U \ \big\vert_{H^s\times H^s}^2 \ + \ \mu  \big\vert \ \partial_x U \ \big\vert_{H^s\times H^s}^2, \]
provided that $\epsilon  \underline{U} $ is sufficiently small, (for the nonlinear terms to be controlled).
\medskip

The following Propositions are a direct consequence from~\cite[Proposition 2.6 and 2.8]{Duchene11a} (where the author deal with the case of internal waves with a free surface, so that the system has four equations, but possess the exact same structure), and we do not detail the proof further on.
\begin{Theorem}[Well-posedness]\label{WPBouss}
Let $U^0\in H^{s+1}(\RR)^2$, with $s>3/2$. Then there exists a constant ${C_0=C(\delta_{\min}^{-1},\delta_{\max},M,\bo_{\min}^{-1})>0}$ and $\epsilon_0=(C_0 \big\vert U^0\big\vert_{X^{s+1}_\mu})^{-1}$ such that for any $0\leq \epsilon \leq \epsilon_0$, there exists ${T>0}$, independent of $\epsilon$, and a unique solution ${U\in C^0([0,T/\epsilon);X^{s+1}_\mu)\cap C^1([0,T/\epsilon);X^{s}_\mu)}$ of the Cauchy problem~\eqref{eqn:BoussinesqSym} with ${U\id{t=0}=U^0}$. 
 
 Moreover, one has the following estimate for $t\in[0,T/\epsilon]$:
\begin{equation}\label{EstWP}
\big| U\big|_{L^\infty([0,t];X^{s+1}_\mu)}+ \big| \partial_t U\big|_{L^\infty([0,t] ; X^{s}_\mu)} \leq C_0  \frac{\big| U^0\big|_{X^{s+1}_\mu}}{1-C_0  \big| U^0\big|_{X^{s+1}_\mu} \epsilon t}.
\end{equation}
\end{Theorem}
\noindent A stability result similar to Proposition~\ref{prop:stability} holds, so that the symmetric Boussinesq system is well-posed in the sense of Hadamard.
\begin{Theorem}[Convergence]
Let $\p\in \P_{LW}$ and $U^0\equiv(\zeta^0,\psi^0)^\top\in H^{s+N}(\RR)^2$, with $N$ sufficiently large, satisfy the hypotheses of Theorem~5 in~\cite{Lannes13}; see Theorem~\ref{th:WPEuler}. Then there exists $C,T,\epsilon_0>0$, independent of $\p\in\P_{LW}$ and $\Bo^{-1}\leq \mu\bo_{\min}^{-1}$ , such that for any $0\leq \epsilon\leq \epsilon_0$,
\begin{itemize}
\item There exists a unique solution $U\equiv (\zeta,\psi)^\top$ to the full Euler system~\eqref{eqn:EulerCompletAdim}, with $\beta=0$ and $d=1$, defined on $[0,T]$ and with initial data $(\zeta^0,\psi^0)^\top$ (provided by Theorem~5 in~\cite{Lannes13});
\item The non-vanishing depth condition is satisfied for $\epsilon\leq \epsilon_0 $, so that one can define $ v\equiv \b u_2-\gamma \b u_1$, with $\b u_1,\b u_2$ as in Definition~\ref{defU1U2};
\item There exists a unique solution $U_B\equiv (\zeta_B,v_B)^\top$ to the symmetric Boussinesq model~\eqref{eqn:BoussinesqSym}, defined on $[0,T]$ and with initial data $(\zeta^0, v\id{t=0})^\top$ (provided by Theorem~\ref{WPBouss});
\item The difference between the two solutions is controlled as
\[ \big\vert (\zeta, v)-(\zeta_B,v_B) \big\vert_{L^\infty([0,T];X^{s+1}_\mu)}\leq C \mu^2 t.\]
\end{itemize}
The above results hold on time interval $t\in[0,T/\epsilon]$ with $T$ bounded by below, independently of $\p\in\P_{LW}$, provided that a stronger criterion is satisfied by the initial data. This corresponds to setting $\varrho=1$ in Theorem~\ref{th:WPEuler}; see criterion (5.5) and Theorem~6 in~\cite{Lannes13} for the precise statement.
\end{Theorem}

\subsection{Scalar models}\label{sec:scalar}
In this section, we are interested in the justification of scalar asymptotic models for the propagation of internal waves (as opposed to all aforementioned models, which consist in a system of evolution equations). The derivation and study of such models have a very rich and ancient history, starting with the work of Boussinesq\cite{Boussinesq71} and Korteweg-de Vries\cite{KortewegDe95} which introduced the famous {\em Korteweg-de Vries equation}
\[
\partial_t u \ + \ c\partial_x u \ + \ \alpha u\partial_x u \ + \ \nu\partial_x^3 u \ = \ 0,
\]
 as a models for the propagation of surface gravity waves, in the long wave regime. However, the complete rigorous justification of such model is much more recent~\cite{KanoNishida86,SchneiderWayne00,BonaColinLannes05} (see~\cite{Duchene11a} for the bi-fluidic case). 
 
 One obvious discrepancy between scalar models such as the KdV equation and coupled models, is that the former selects a direction of propagation (right or left, depending on the sign of $c$). On the contrary, the first order approximation of coupled models is a linear wave equation, which predicts that any initial perturbation of the flow will split into two counter-propagating waves. This yields two very different possible justifications of scalar models:
 \begin{itemize}
\item {\em Unidirectional approximation.} One proves that if the initial perturbation (that is the deformation of the interface as well as shear layer-mean velocity) is carefully chosen, then the flow can be approximated as a solution of a scalar equation. Physically speaking, this means that we focus our attention on only one of the two counter-propagating waves after they have split.
\item {\em Decouled approximation.} For a generic initial perturbation of the flow, we approximate the flow as the superposition of two (uncoupled) waves, each one driven by a scalar equation. This justification is of course more general as for the admissible initial data, but its precision is often much worse, as controlling the coupling effects between the two counter-propagating waves (and in particular their secular growth), which are neglected in the decoupled approximation, is arduous. 
 \end{itemize}
 
A major difference between the water-wave case and the bi-fluidic case is that in the latter, there exists a {\em critical ratio} ($\delta^2=\gamma$) for which the first order (quadratic) nonlinearity of our models vanishes; see the Boussinesq system~\eqref{eqn:BoussinesqMean} for example. This allows to consider a regime with greater nonlinearities than the long wave regime (and in particular the Camassa-Holm regime; see Definition~\ref{Def:Regimes}), and thus motivates higher order models than the Korteweg-de Vries equations, in order to recover a $\O(\mu^2)$ precision. Here, we focus on the so-called {\em Constantin-Lannes} equation:
\begin{multline} \label{eq:CLI}
 (1- \ \mu\beta \partial_x^2)\partial_t v \ +\ \epsilon \alpha_1 v\partial_x v \ + \ \epsilon^2 \alpha_2 v^2\partial_x v\ + \ \epsilon^3 \alpha_3 v^3\partial_x v \\
 +\ \mu\nu \partial_x^3 v \ \ + \ \mu\epsilon\partial_x\big(\kappa_1 v\partial_x^2 v+\kappa_2(\partial_x v)^2\big) \ = \ 0, 
\end{multline}
where $\beta,\alpha_i$ ($i=1,2,3$), $\nu$, $\kappa_1$, $\kappa_2$, are fixed, given parameters . 
This equation has been studied and justified as a model for the propagation of {\em unidirectional} surface gravity waves in~\cite{ConstantinLannes09}. In particular, the well-posedness of the Cauchy problem for~\eqref{eq:CLI} in Sobolev spaces is proved, provided $\beta>0$.

The justification of both the unidirectional and decoupled approximation in the bi-fluidic case, in the sense of consistency, have been worked out in~\cite{Duchene13}. In what follows, we restrict to the Camassa-Holm regime and to the Constantin-Lannes equations for the sake of simplicity, but more general results may be found therein. The full justification, and in particular the convergence results stated below is then a consequence of the properties of the Green-Naghdi type model introduced by the authors in~\cite{DucheneIsrawiTalhouk}, and recalled in section~\ref{sec:Serre}. We let the reader refer to the aforementioned references for more details, and disclose the statements below without further comments.

\begin{Definition}[Unidirectional approximation]
Let $\zeta^0\in H^s(\RR)$, $s>5/2$, $(\theta,\lambda)\in \RR^2$, and set parameters $\p=(\epsilon,\mu,\gamma,\delta)\in\P_{CH}$, as defined in Definition~\eqref{Def:Regimes}. Then the {\em Constantin-Lannes unidirectional approximation}, $(\zeta_{\text{u}},v_{\text{u}})$, is defined as follows.

 Let $\zeta_{\text{u}}$ be the unique solution of 
\begin{multline} \label{eq:CLu}
\partial_t \zeta + \partial_x \zeta + \epsilon\alpha_1\zeta\partial_x \zeta + \epsilon^2 \alpha_2 \zeta^2\partial_x \zeta +\epsilon^3 \alpha_3 \zeta^3\partial_x \zeta+\mu\nu_x^{\theta,\lambda}\partial_x^3\zeta-\mu\nu_t^{\theta,\lambda}\partial_x^2\partial_t\zeta \\
+\mu\epsilon\partial_x\left(\kappa_1^{\theta,\lambda} \zeta\partial_x^2\zeta +\kappa_2^{\theta} (\partial_x \zeta)^2\right) \ =0\ ,
 \end{multline}
with ${\zeta_{\text{u}}}\id{t=0}=\zeta^0$, and parameters 
\[ \begin{array}{l}
\alpha_1 =  \frac32\frac{\delta^2-\gamma}{\gamma+\delta}, \
\alpha_2 =  \frac{21(\delta^2-\gamma)^2}{8(\gamma+\delta)^2}-3\frac{\delta^3+\gamma}{\gamma+\delta},
\ \alpha_3 =  \frac{71(\delta^2-\gamma)^3}{16(\gamma+\delta)^3} -\frac{37(\delta^2-\gamma)(\delta^3+\gamma)}{4(\gamma+\delta)^2}+\frac{5(\delta^4-\gamma)}{\gamma+\delta}, \\
 \nu_x^{\theta,\lambda} = \ (1-\theta-\lambda)\frac{1+\gamma\delta}{6\delta(\gamma+\delta)} , \quad
 \nu_t^{\theta,\lambda} = \ (\theta+\lambda)\frac{1+\gamma\delta}{6\delta(\gamma+\delta)} ,\\
\kappa_1^{\theta,\lambda} =  \frac{(14-6(\theta+\lambda))(\delta^2-\gamma)(1+\gamma\delta)}{24\delta(\gamma+\delta)^2}-\frac{1-\gamma}{6(\gamma+\delta)}, \quad
\kappa_2^{\theta} =  \frac{(17-12\theta)(\delta^2-\gamma)(1+\gamma\delta)}{48\delta(\gamma+\delta)^2}-\frac{1-\gamma}{12(\gamma+\delta)}.
\end{array}
\]
It is assumed that $\theta,\lambda$ are chosen such that $\nu_t^{\theta,\lambda}>0$, so that~\eqref{eq:CLu} is well-posed, and $\zeta(t,\cdot)\in H^s$ is uniquely defined over time scale $\O(1/\epsilon)$; see~\cite[Proposition 4]{ConstantinLannes09}. 

Then define
$v_{\text{u}}$ as $ v_{\text{u}}=\frac{h_1+\gamma h_2}{h_1h_2}\underline{v}[\zeta_{\text{u}}]$, with
\begin{equation}\label{ztov}
\underline{v}[\zeta] \ = \ \zeta + \epsilon\frac{\alpha_1}2\zeta^2 + \epsilon^2 \frac{\alpha_2}{3} \zeta^3 +\epsilon^3 \frac{\alpha_3}4 \zeta^4 +\mu\nu\partial_x^2\zeta+\mu\epsilon\left(\kappa_1 \zeta\partial_x^2\zeta +\kappa_2 (\partial_x \zeta)^2\right),
\end{equation}
where parameters $\alpha_1,\alpha_2,\alpha_3$ are as above, and
$ \nu=\nu^{0,0}_x$, $ \kappa_1=\kappa_1^{0,0}$, $ \kappa_2=\kappa_2^{0}$.
\end{Definition}
\begin{Theorem}[Convergence of the unidirectional approximation] \label{prop:unidir}
Let $\p\in \P_{CH}$ and $U^0\equiv(\zeta^0,\psi^0)^\top\in H^{s+N}(\RR)^2$, with $N$ sufficiently large, satisfy the hypotheses of~\cite[Theorem~5]{Lannes13} (see Theorem~\ref{th:WPEuler}) as well as the ones of Theorem~\ref{thbi1mr}. Define $v^0\equiv \b u_2^0-\gamma \b u_1^0$, with $\b u_1^0,\b u_2^0$ as in Definition~\ref{defU1U2}, and assume that  $(\zeta^0,v^0)^\top$ satisfies~\eqref{ztov}. Then there exists $C,T>0$, independent of $\p$, such that
\begin{itemize}
\item There exists a unique solution $U\equiv (\zeta,\psi)^\top$ to the full Euler system~\eqref{eqn:EulerCompletAdim}, with $\beta=0$ and $d=1$, defined on $[0,T]$ and with initial data $(\zeta^0,\psi^0)^\top$ (provided by Theorem~5 in~\cite{Lannes13}); We denote $v=\b u_2-\gamma \b u_1$ the shear layer-mean velocity with $\b u_1,\b u_2$ as in Definition~\ref{defU1U2}. 
\item The Constantin-Lannes unidirectional approximation $(\zeta_{u},v_{\text{u}})^\top$ with initial data $(\zeta^0,v^0)^\top$ is uniquely defined for $t\in [0,T]$ as described above.
\item The difference between the two solutions is controlled as
\[ \big\vert (\zeta,v)-(\zeta_{\text{u}},v_{\text{u}}) \big\vert_{L^\infty([0,T];H^s)^2} \leq C (\mu^2+\Bo^{-1}) t.\]
\end{itemize}
If the initial data satisfies a stronger stability hypothesis (see~\cite[Theorem~6]{Lannes13}), then the result holds for large time $t\in [0,T/\epsilon]$, with $T$ independent of $\p$.
\end{Theorem}

\begin{Definition}[Decoupled approximation] \label{Def:CL}
Let $(\zeta^0,v^0)^\top\in H^s(\RR)^2$, $s>5/2$, and set parameters $(\epsilon,\mu,\gamma,\delta)\in\P_{CH}$, as defined in Definition~\ref{Def:Regimes}, and $(\lambda,\theta)\in\RR^2$. The {\em Constantin-Lannes decoupled approximation} is then
\[ U_{\text{CL}}\ \equiv \ \Big(v_+(t,x-t)+v_-(t,x+t),(\gamma+\delta)\big(v_+(t,x-t)-v_-(t,x+t)\big)\Big),\]
 where ${v_\pm}\id{t=0} \ = \ \frac12(\zeta^0\pm\frac{v^0}{\gamma+\delta})\id{t=0}$ and $v_\pm=(1\pm\mu\lambda \partial_x^2)^{-1}v_\pm^\lambda$ where $v_\pm^\lambda$ satisfies
\begin{multline} \label{eq:CL}
 \partial_t v_\pm^\lambda \ \pm \ \epsilon\alpha_1 v_\pm^\lambda\partial_x v_\pm^\lambda \ \pm \ \epsilon^2 \alpha_2 (v_\pm^\lambda)^2\partial_x v_\pm^\lambda\ \pm \ \epsilon^3 \alpha_3^{\theta,\lambda}(v_\pm^\lambda)^3\partial_x v_\pm^\lambda \\
\pm\ \mu\nu^{\theta,\lambda}_x \partial_x^3 v_\pm^\lambda \ - \ \mu\nu^{\theta,\lambda}_t \partial_x^2\partial_t v_\pm^\lambda \ \pm \ \mu\epsilon\partial_x\big(\kappa_1^{\theta,\lambda} v_\pm^\lambda\partial_x^2 v_\pm^\lambda+\kappa_2^{\theta}(\partial_x v_\pm^\lambda)^2\big) \ = \ 0, 
\end{multline}
with parameters given  by
\[ %\begin{equation} \label{eqn:parameters}
\begin{array}{c} 
\alpha_1=\frac32\frac{\delta^2-\gamma}{\gamma+\delta},\quad \alpha_2=-3\frac{\gamma\delta(\delta+1)^2}{(\gamma+\delta)^2},\quad \alpha_3=-5\frac{\delta^2(\delta+1)^2\gamma(1-\gamma)}{(\gamma+\delta)^3},\\ 
 \nu_t^{\theta,\lambda}\equiv \frac\theta6\frac{1+\gamma\delta}{\delta(\delta+\gamma)} + \lambda , \qquad \nu_x^{\theta,\lambda}\equiv \frac{1-\theta}6\frac{1+\gamma\delta}{\delta(\delta+\gamma)} -\lambda,\\
 \kappa_1^{\theta,\lambda} \equiv \frac{(1+\gamma\delta)(\delta^2-\gamma)}{3\delta(\gamma+\delta)^2}(1+\frac{1-\theta}4)-\frac{(1-\gamma)}{6(\gamma+\delta)}+\lambda\frac32\frac{\delta^2-\gamma}{\gamma+\delta} , \quad
 \kappa_2^\theta \equiv \frac{(1+\gamma\delta)(\delta^2-\gamma)}{3\delta(\gamma+\delta)^2}(1+\frac{1-\theta}4)-\frac{(1-\gamma)}{12(\gamma+\delta)}.\end{array}
\] %\end{equation}
As above, it is assumed that $\theta,\lambda$ are chosen such that $\nu_t^{\theta,\lambda}>0$, thus~\eqref{eq:CL} is well-posed.
\end{Definition}

\begin{Theorem}[Convergence of the decoupled approximation]\label{prop:decoupled} 
Let $\p\in \P_{CH},\ s\geq0$ and let $U^0\equiv(\zeta^0,\psi^0)^\top\in H^{s+N}(\RR)^2$, with $N$ sufficiently large, satisfy the hypotheses of~\cite[Theorem~5]{Lannes13} (see Theorem~\ref{th:WPEuler}) as well as the ones of Theorem~\ref{thbi1mr}. Then there exists $C,T>0$, independent of $\p$, such that 
\begin{itemize}
\item There exists a unique solution $U\equiv (\zeta,\psi)^\top$ to the full Euler system~\eqref{eqn:EulerCompletAdim}, with $\beta=0$ and $d=1$, defined on $[0,T]$ and with initial data $(\zeta^0,\psi^0)^\top$ (provided by Theorem~5 in~\cite{Lannes13}). We denote $v=\b u_2-\gamma \b u_1$ the shear layer-mean velocity with $\b u_1,\b u_2$ as in Definition~\ref{defU1U2};
\item The Constantin-Lannes decoupled approximation $U_{CL}\equiv (\zeta_{\text{d}},v_{\text{d}})^\top$ with initial data $(\zeta^0,v\id{t=0})^\top$ is uniquely defined for $t\in [0,T]$ as described above.
\item The difference between the two solutions is controlled as
\[ \big\vert (\zeta,v)-(\zeta_{\text{d}},v_{\text{d}}) \big\vert_{L^\infty([0,T];H^s)^2} \leq \ C\ \big( \varepsilon_0 \ \min(t,t^{1/2}) (1 \ + \ \varepsilon_0 t ) \ + \ \Bo^{-1}t \big),\]
with $\varepsilon_0 = \max\{\epsilon(\delta^2-\gamma),\mu\}$.
\item Moreover, assume that the initial data is sufficiently localized in space, that is more precisely $(1+|\cdot |^2) \partial^k\zeta^0$ and $(1+|\cdot |^2)\partial^kv\id{t=0}\in H^{s}(\RR)$, $k\in\{0,\dots,7\}$, then one has the improved estimate
\[  \big\vert (\zeta,v)-(\zeta_{\text{d}},v_{\text{d}}) \big\vert_{L^\infty([0,T];H^s)^2}  \leq \ C\ \big(\varepsilon_0 \ \min(t,1) (1 \ + \ \varepsilon_0 t ) \ + \ \Bo^{-1}t\big).\]
\end{itemize}
The first three items hold for large time $t\in [0,T/\epsilon]$, with $T$ independent of $\p$, provided that a stronger criterion is satisfied by the initial data (see hypotheses in~\cite[Theorem 6]{Lannes13}). In that case, the last item is valid for time $t\in [0,T/\max\{\epsilon,\mu\}]$.
\end{Theorem}

\end{document}